\documentclass{mcom-l}
\usepackage{mathrsfs}
\usepackage{hyperref}
\usepackage{amssymb,amsmath,amsthm}
\usepackage{multirow}
\usepackage{graphicx,subfigure,epstopdf}

\newtheorem{theorem}{Theorem}[section]
\newtheorem{lemma}[theorem]{Lemma}

\newtheorem{corollary}[theorem]{Corollary}

\theoremstyle{definition}
\newtheorem{definition}[theorem]{Definition}
\newtheorem{example}[theorem]{Example}

\theoremstyle{remark}
\newtheorem{remark}[theorem]{Remark}

\newcommand{\e}{\mathrm{e}}
\numberwithin{equation}{section}

\begin{document}

\title[Adaptive DG methods for fractional PDEs]{Discontinuous Galerkin methods and their adaptivity for the tempered fractional (convection) diffusion equations}

\author{Xudong Wang}
\address{School of Mathematics and Statistics, Gansu Key Laboratory of Applied Mathematics and Complex Systems, Lanzhou University, Lanzhou 730000, P.R. China}
\email{xdwang14@lzu.edu.cn}
\thanks{}


\author{Weihua Deng}
\address{School of Mathematics and Statistics, Gansu Key Laboratory of Applied Mathematics and Complex Systems, Lanzhou University, Lanzhou 730000, P.R. China}
\email{dengwh@lzu.edu.cn}
\thanks{}


\date{}

\dedicatory{}

\begin{abstract}
This paper focuses on the adaptive discontinuous Galerkin (DG) methods for the tempered fractional (convection) diffusion equations. The DG schemes with interior penalty for the diffusion term and numerical flux for the convection term are used to solve the equations, and the detailed stability and convergence analyses are provided. Based on the derived posteriori error estimates, the local error indicator is designed. The theoretical results and the effectiveness of the adaptive DG methods are respectively verified and displayed by the extensive numerical experiments. The strategy of designing adaptive schemes presented in this paper works for the general PDEs with fractional operators.
\end{abstract}

\maketitle

\section{Introduction}\label{section:1}
Fractional calculus \cite{Butzer1} is a popular mathematical tool for modeling anomalous diffusions \cite{Metzler1}, being ubiquitous in nature. Microscopically, anomalous diffusion can be described by continuous time random walk (CTRW), governed by the waiting time and jump length; generally the first moment of the waiting time and/or the second moment of the jump length diverge(s). Sometimes, it is better to temper the broad distribution(s) of the waiting time and/or the jump length \cite{Baenmera1,Deng4,Li2}, because of the boundedness of physical space or the finite lifespan of the biological particles or the slow transition of different diffusion types. Based on the tempered CTRW, the partial differential equations (PDEs) characterizing the evolution of the functional distribution of the trajectories of the particles are derived \cite{Wu1}, which reduce to the PDEs describing the distribution of the positions of the particles if taking the parameter $p$ over there as $0$, called tempered fractional PDEs; here, we discuss their (adaptive) discontinuous Galerkin (DG) methods.


There are already some progresses for numerically solving (tempered) fractional PDEs by variational methods \cite{Deng4, Ervin1, Huang1, LiXu, Mclean, Qiu1, Wang1, Xu1, Zayernouri}. Ervin and Roop \cite{Ervin1} firstly present the variational formulation for the fractional advection dispersion equation. The DG methods are particularly extended to fractional problems with their majority of characteristics \cite{Erik1, Shu1, Hesthaven1, Riv1, WangWang, Shu2}, naturally being formulated for any order of accuracy in any element, being flexible in choosing element sizes in any place, suitable for adaptivity, being local and easy to invert for mass matrix, leading to an explicit formulation for time dependent problems, etc.  Cockburn and Mustapha \cite{Cockburn} provide a hybridizable DG method for fractional diffusion problems; Mclean and Mustapha \cite{Mclean} discuss the superconvergence of the DG method for the fractional diffusion and wave equations; Xu and Hesthaven \cite{Xu1}, and  Wang et al \cite{Wang1}, respectively, consider DG and hybridized DG methods for the fractional convection-diffusion equations; Zayernouri and Karniadakis \cite{Zayernouri} design discontinuous spectral element methods for the time and space fractional differential equations. Du et al \cite{Du1} give a convergent adaptive finite element algorithm for nonlocal diffusion and peridynamic models. It seems that there are not works for digging out the potential advantages of DG methods in adaptivity for fractional problems, by deriving posteriori error estimates and providing the local error indicators.

The model we consider in this paper is the two dimensional space tempered fractional differential equation with absorbing boundary conditions \cite{Dybiec:06}, i.e.,
\begin{equation}\label{equation:1.1}
\left\{ \begin{array}{ll}
\partial_t u+\textbf{b}\cdot\nabla u -\kappa_1\nabla_x^{\alpha,\lambda}u-\kappa_2\nabla_y^{\beta,\lambda}u=f,
& (\textbf{x},t)\in\Omega\times J, \\
u(\textbf{x},0)=u_0(\textbf{x}), &  \textbf{x}\in\Omega, \\
u(\textbf{x},t)=0,   &  (\textbf{x},t)\in\mathbb{R}^2\backslash\Omega\times J,
\end{array} \right.
\end{equation}
where $\alpha,\beta\in(0,2)\backslash\{1\}$, $\lambda>0$, and $\kappa_1,\kappa_2>0$
in the domain $\Omega=[a,b]\times[c,d]$ and $J=[0,T]$; the function $f\in L^2(J;L^2(\Omega))$ is a source term; the convection coefficient $\textbf{b}$ is a given divergence-free velocity field, i.e., $\nabla\cdot\mathbf{b}=0$, being supposed to satisfy $\textbf{b}\in L^\infty(J;W^{1,\infty}(\Omega)^2)$, and the initial function $u_0\in L^2(\Omega)$. As for the discussion of the adaptivity of the fractional problems, we start from the steady state version of (\ref{equation:1.1}) with $\textbf{b}=0$. The first part of the paper focuses on designing the DG scheme of Eq. (\ref{equation:1.1}) with genuinely unstructured grids, and offering explicit theoretical analyses. Being different from \cite{Qiu1}, which constructs the LDG scheme by rewriting the fractional equation as a first order system, we adopt the primal DG methods, namely interior penalty (IP) method, still keeping the advantages over the classical continuous Galerkin method in facilitating $hp$-adaptivity and yielding block diagonal mass matrices in time-dependent problems. Generally, the non-ignorable drawback of the IP method is to specify sufficient large penalty parameter for guaranteeing numerical stability, which degrades the performance of the iterative solver of the linear system \cite{Shahbazi1}. Fortunately, for the (tempered) fractional equations, this weak point disappears, since the schemes are stable for any value of the penalty parameter, say, simply taking as $1$. For the convection term, the upwind flux \cite{Shu1,Riv1} is used for ensuring numerical stability.


Mesh adaption is the basic technique of balancing the computational cost and accuracy, which introduces extra points near the singularities or the high gradient part of the solution to be computed. The key ingredient of adaptivity is a posteriori error estimators \cite{Erik1, Ming1}, which are computable quantities depending on the computed solution and data. We derive a posteriori error estimators for fractional operators and obtain the local error indicators, being used to  dynamically and locally refine or coarsen meshes. To show the effectiveness of the local error indicators, we adaptively solve the fractional differential equations with singularities, including both the steady state and time dependent ones. For the steady state equations, two schemes are presented. One is based on energy norm, while another one is based on dual weighted residual. It is observed that the latter performs better than the former. For the time dependent one, both the space mesh and time-step size are adapted simultaneously.

The outline of this paper is as follows. Section 2 is composed of five subsections. The first subsection reviews the definitions and properties of tempered fractional calculus. The notations and the variational formulations of DG schemes are, respectively, proposed in the second and third subsections. In the fourth subsection, we perform the stability analysis and error estimates for the two dimensional tempered fractional convection-diffusion equations. The numerical results are provided in the last subsection. Section 3 is discussing adaptivity, composed of two parts, which are, respectively, for the stationary equation and the evolution equation. We conclude the paper with some remarks in the last section.

\section{DG for tempered fractional convection-diffusion equation}\label{section:2}

In this section, we design the DG scheme for the tempered fractional convection-diffusion equation (\ref{equation:1.1}), provide the detailed stability and convergence proof, and numerically verify the theoretical results.

\subsection{Tempered fractional operators}\label{subsection:2.1}
We firstly introduce some preliminary definitions of tempered fractional calculus \cite{Deng4,Li2}.

\begin{definition}\label{define:2.1.1}
For any $\alpha>0, \lambda>0$, the left and right tempered Riemann-Liouville fractional integrals of function $u(x)$
defined on $\mathbb{R}$ are given by
\begin{equation}\label{equation:2.1.1}
_{-\infty}I_x^{\alpha,\lambda}u(x)={\e^{-\lambda x}} _{-\infty}I_x^{\alpha}[\e^{\lambda x}u(x)]
=\frac{1}{\Gamma(\alpha)}\int_{-\infty}^x(x-\xi)^{\alpha-1}\e^{-\lambda(x-\xi)}u(\xi)d\xi,
\end{equation}
and
\begin{equation}\label{equation:2.1.2}
_xI_{\infty}^{\alpha,\lambda}u(x)={\e^{\lambda x}} _xI_{\infty}^{\alpha}[\e^{-\lambda x}u(x)]
=\frac{1}{\Gamma(\alpha)}\int_x^{\infty}(\xi-x)^{\alpha-1}\e^{-\lambda(\xi-x)}u(\xi)d\xi.
\end{equation}
\end{definition}

\begin{definition}\label{define:2.1.2}
For any $\alpha>0, n-1<\alpha<n, n\in \mathbb{N}^+, \lambda>0,$ the left and right tempered Riemann-Liouville fractional derivatives of function $u(x)$ defined on $\mathbb{R}$ are given as
\begin{equation}\label{equation:2.1.3}
_{-\infty}D_x^{\alpha,\lambda}u(x)={\e^{-\lambda x}} _{-\infty}D_x^{\alpha}[\e^{\lambda x}u(x)]
={(\lambda+D)^n}_{-\infty}I_x^{n-\alpha,\lambda}u(x),
\end{equation}
and
\begin{equation}\label{equation:2.1.4}
_xD_{\infty}^{\alpha,\lambda}u(x)={\e^{\lambda x}} _xD_{\infty}^{\alpha}[\e^{-\lambda x}u(x)]
={(\lambda-D)^n}_xI_{\infty}^{n-\alpha,\lambda}u(x).
\end{equation}
\end{definition}

\begin{definition}\label{define:2.1.2,+}
For any $\alpha>0, n-1<\alpha<n, n\in \mathbb{N}^+, \lambda>0,$ the left and right tempered Caputo fractional derivatives of function $u(x)$ defined on $\mathbb{R}$ are described by
\begin{equation}\label{equation:2.1.3,+}
{ {}_-\hspace{0cm}_\infty^CD_x^{\alpha, \lambda} } u(x)
=\e^{-\lambda x}{{}_-\hspace{0cm}_\infty^CD_x^\alpha}[\e^{\lambda x}u(x)]
={}_{-\infty}I_x^{n-\alpha,\lambda}{(\lambda+D)^n}u(x),
\end{equation}
and
\begin{equation}\label{equation:2.1.4,+}
{_x^C}D_{\infty}^{\alpha,\lambda}u(x)=\e^{\lambda x}{_x^C}D_{\infty}^{\alpha}[\e^{-\lambda x}u(x)]
={_x}I_{\infty}^{n-\alpha,\lambda}{(\lambda-D)^n}u(x).
\end{equation}
\end{definition}

\begin{definition}\label{define:2.1.3}
The Riesz tempered fractional derivatives $\nabla_x^{\alpha,\lambda}$ and $\nabla_y^{\beta,\lambda}$ with $\alpha,\beta\in(0,2)\backslash\{1\}$, $\lambda>0$, are respectively defined as
\begin{equation}\label{equation:2.1.5}
\nabla_x^{\alpha,\lambda}u(\textbf{x},t)=-\kappa_\alpha[_{-\infty}\nabla_x^{\alpha,\lambda}+{}_x\nabla_{\infty}^{\alpha,\lambda}]
u(\textbf{x},t),
\end{equation}
and
\begin{equation}\label{equation:2.1.6}
\nabla_y^{\beta,\lambda}u(\textbf{x},t)=-\kappa_\beta[_{-\infty}\nabla_y^{\beta,\lambda}+{}_y\nabla_{\infty}^{\beta,\lambda}]
u(\textbf{x},t),
\end{equation}
where $\kappa_\alpha=\frac{1}{2\cos(\alpha\pi/2)}$ and $\kappa_\beta=\frac{1}{2\cos(\beta\pi/2)}$. The left and right tempered Riemann-Liouville fractional derivatives are defined by
\begin{equation}\label{equation:2.1.7}
_{-\infty}\nabla_x^{\alpha,\lambda}u(\textbf{x},t)= {_{-\infty}D_x^{\alpha,\lambda}}u(\textbf{x},t)-\lambda^\alpha u(\textbf{x},t)~~ {\rm for}~~ \alpha \in (0,1),
\end{equation}
\begin{equation}\label{equation:2.1.7.1}
_{-\infty}\nabla_x^{\alpha,\lambda}u(\textbf{x},t)= {_{-\infty}D_x^{\alpha,\lambda}}u(\textbf{x},t)-\lambda^\alpha u(\textbf{x},t)
-\alpha\lambda^{\alpha-1}\frac{\partial u(\textbf{x},t)}{\partial x}~~ {\rm for}~~ \alpha \in (1,2),
\end{equation}
and
\begin{equation}\label{equation:2.1.8}
_x\nabla_{\infty}^{\alpha,\lambda}u(\textbf{x},t)= {_xD_{\infty}^{\alpha,\lambda}}u(\textbf{x},t)-\lambda^\alpha u(\textbf{x},t)~~ {\rm for}~~ \alpha \in (0,1),
\end{equation}
\begin{equation}\label{equation:2.1.8.1}
_x\nabla_{\infty}^{\alpha,\lambda}u(\textbf{x},t)= {_xD_{\infty}^{\alpha,\lambda}}u(\textbf{x},t)-\lambda^\alpha u(\textbf{x},t)
+\alpha\lambda^{\alpha-1}\frac{\partial u(\textbf{x},t)}{\partial x}~~ {\rm for}~~ \alpha \in (1,2),
\end{equation}
respectively. The definition of $\nabla_y^{\beta,\lambda}$ is similar.
\end{definition}

\begin{remark}\label{remark:2.1.1}
For $\lambda=0$, the tempered Riemann-Liouville fractional derivative reduces to the Riemann-Liouville fractional derivative and the tempered Riesz fractional derivative to the Riesz fractional derivative.
The fractional substantial derivative \cite{Chen1,Huang1} has similar definition to $_{-\infty}D_x^{\alpha,\lambda}$, but their physical backgrounds are totally different.
\end{remark}

Because of the absorbing boundary condition of Eq. (\ref{equation:1.1}), and Definitions  \ref{define:2.1.1} and \ref{define:2.1.2}, we have $_{-\infty}I_x^{\alpha,\lambda}u(x)={}_aI_x^{\alpha,\lambda}u(x)$ and $_xI_{\infty}^{\alpha,\lambda}u(x)={}_xI_b^{\alpha,\lambda}u(x)$; it is similar for the tempered fractional derivative. For convenience, we use the latter notations in the following.
First we introduce some properties of the tempered fractional calculus and the tempered fractional spaces. 
Suppose that the function $f(x)$ is $(m-1)$-times continuously differentiable in the interval $[a,b]$ and that its $m$-times derivative is integrable in $[a,b]$. Then for any $m-1<\mu,\nu<m, \lambda>0$,
\begin{equation}\label{2.2.1.+}
  {_a}D_x^{\mu,\lambda}f(x)={_a^C}D_x^{\mu,\lambda}f(x)+
  \sum_{j=0}^{m-1}D^{j,\lambda}f(x)|_{x=a}\frac{(x-a)^{j-\mu}\e^{-\lambda(x-a)}}{\Gamma(j-\mu+1)},
\end{equation}
\begin{equation}\label{2.2.2.+}
  {_a}I_x^{\mu,\lambda}[{_a}D_x^{\mu,\lambda}f(x)]=f(x)-
  \sum_{j=1}^m{_a}D_x^{\mu-j,\lambda}f(x)|_{x=a}\frac{(x-a)^{\mu-j}\e^{-\lambda(x-a)}}{\Gamma(\mu-j+1)},
\end{equation}
\begin{equation}\label{2.2.3.+}
  {_a}D_x^{\nu,\lambda}[{_a}D_x^{\mu,\lambda}f(x)]={_a}D_x^{\nu+\mu,\lambda}f(x)-
  \sum_{j=1}^m{_a}D_x^{\mu-j,\lambda}f(x)|_{x=a}\frac{(x-a)^{-\nu-j}\e^{-\lambda(x-a)}}{\Gamma(-\nu-j+1)},
\end{equation}
where $D^{j,\lambda}$ denotes $(\lambda+D)^j$ when $j$ is an integer. It can be noted that the difference between tempered Riemann-Liouville fractional derivatives and tempered Caputo fractional derivatives is the sum of the values of the derivatives of $f(x)$ at boundary. So, we state the following condition.

\noindent\textbf{Condition A:} For $m-1<\mu<m$, the function $f(x)$ satisfies \[f^{(j)}(a)=f^{(j)}(b)=0 \quad {\rm for~~} j=0,1,\cdots,m-1.\]

From the discussion of (\cite{Podlubny1}, p. 75-77), we know that if $f(x)$ is sufficiently smooth, then the conditions
\begin{equation}
  {_a}D_x^{\mu-j,\lambda}f(x)|_{x=a}={_x}D_b^{\mu-j,\lambda}f(x)|_{x=b}=0 \qquad {\rm for~~} j=1,2,\cdots,m,
\end{equation}
are equivalent to Condition A.
Therefore, under Condition A, the formulae (\ref{2.2.1.+}), (\ref{2.2.2.+}) and (\ref{2.2.3.+}) become
\begin{equation}\label{2.2.4.+}
  {_a}D_x^{\mu,\lambda}f(x)={_a^C}D_x^{\mu,\lambda}f(x),
\end{equation}
\begin{equation}\label{2.2.5.+}
  {_a}I_x^{\mu,\lambda}[{_a}D_x^{\mu,\lambda}f(x)]=f(x),
\end{equation}
\begin{equation}\label{2.2.6.+}
  {_a}D_x^{\nu,\lambda}[{_a}D_x^{\mu,\lambda}f(x)]={_a}D_x^{\nu+\mu,\lambda}f(x).
\end{equation}

\begin{lemma}[\rm Adjoint property]\label{lemma:2.2.1}
For any $\alpha>0, \lambda>0,$ the left and right tempered Riemann-Liouville fractional integral operators are adjoint for any functions $u(x),v(x)\in L^2([a,b]),$ i.e.,
\begin{equation}\label{equation:2.2.4}
\int_a^b{_aI_x^{\alpha,\lambda}}u(x)\cdot v(x)dx=\int_a^b u(x)\cdot {_xI_b^{\alpha,\lambda}v(x)}dx.\nonumber
\end{equation}
\end{lemma}

\begin{lemma}[\rm Adjoint property]\label{lemma:2.2.2}
For any $\alpha>0, \lambda>0,$ the left and right tempered Riemann-Liouville fractional derivative operators are adjoint for functions $u(x)$ and $v(x)$ under Condition A, i.e.,
\begin{equation}\label{equation:2.2.5}
\int_a^b{_aD_x^{\alpha,\lambda}}u(x)\cdot v(x)dx=\int_a^b u(x)\cdot {_xD_b^{\alpha,\lambda}v(x)}dx.\nonumber
\end{equation}
\end{lemma}

\begin{lemma}\label{lemma:2.2.1.+}
  For $u(x)\in L^2(\mathbb{R})$ and $\mu,\lambda>0$, it holds that
  \begin{equation*}
    \mathscr{F}[_{-\infty}I_x^{\mu,\lambda}u(x)](\omega)=(\lambda+i\omega)^{-\mu}\hat{u}(\omega),
  \end{equation*}
  \begin{equation*}
    \mathscr{F}[_xI_{\infty}^{\mu,\lambda}u(x)](\omega)=(\lambda-i\omega)^{-\mu}\hat{u}(\omega).
  \end{equation*}
  If $u(x)\in C_0^\infty(\mathbb{R})$ further, then
    \begin{equation*}
    \mathscr{F}[_{-\infty}D_x^{\mu,\lambda}u(x)](\omega)=(\lambda+i\omega)^{\mu}\hat{u}(\omega),
  \end{equation*}
  \begin{equation*}
    \mathscr{F}[_xD_{\infty}^{\mu,\lambda}u(x)](\omega)=(\lambda-i\omega)^{\mu}\hat{u}(\omega),
  \end{equation*}
  where $\hat{u}(\omega)=\int_{-\infty}^\infty ~ \e^{-i\omega x}u(x)d x$.
\end{lemma}

Next, we discuss the tempered fractional Sobolev space. Denote $\Omega=[a,b]\times[c,d]$ as a finite domain; $A\lesssim B$ means that $A$ can be bounded by a multiple of $B$, which is independent of $B$; and $A\sim B$ means that $A\lesssim B\lesssim A$. The definitions and properties for the norms of the left and right tempered fractional derivatives are similar, so we mainly focus on the left one.

\begin{lemma}[\cite{Deng4}]
  Let $p,q>0$ and $\alpha>0$. Then
  \begin{equation}\label{equation:2.3.1.+}
    (p+q)^\alpha\thicksim p^\alpha+q^\alpha.
  \end{equation}
  More specifically,
  \begin{equation}
    2^{\alpha-1}(p^\alpha+q^\alpha)\leq(p+q)^\alpha\leq(p^\alpha+q^\alpha)\quad {\rm for }\quad 0<\alpha\leq1\nonumber
  \end{equation}
  and
  \begin{equation}
    (p^\alpha+q^\alpha)\leq(p+q)^\alpha\leq2^{\alpha-1}(p^\alpha+q^\alpha)\quad {\rm for }\quad \alpha>1.\nonumber
  \end{equation}
\end{lemma}

Let $\alpha,\beta>0$ and $u(x,y)\in L^2(\mathbb{R}^2)$. Define the semi-norms
\begin{equation}
  |u(x,y)|_{J_{L,x}^{\alpha,\lambda}(\mathbb{R}^2)}:=\|_{-\infty}D_x^{\alpha,\lambda}u(x,y)\|_{L^2(\mathbb{R}^2)},\nonumber
\end{equation}
\begin{equation}
  |u(x,y)|_{J_{L,y}^{\beta,\lambda}(\mathbb{R}^2)}:=\|_{-\infty}D_y^{\beta,\lambda}u(x,y)\|_{L^2(\mathbb{R}^2)},\nonumber
\end{equation}
and norms
\[ \|u(x,y)\|_{J_{L,x}^{\alpha,\lambda}(\mathbb{R}^2)}:=
(\|u(x,y)\|_{L^2(\mathbb{R}^2)}^2+|u(x,y)|_{J_{L,x}^{\alpha,\lambda}(\mathbb{R}^2)}^2)^{1/2}, \]
\[ \|u(x,y)\|_{J_{L,y}^{\beta,\lambda}(\mathbb{R}^2)}:=
(\|u(x,y)\|_{L^2(\mathbb{R}^2)}^2+|u(x,y)|_{J_{L,y}^{\beta,\lambda}(\mathbb{R}^2)}^2)^{1/2}. \]

Next, we define the norms for the functions in $H^{\mu,\lambda}(\mathbb{R}^2)$ in terms of the Fourier transform, i.e.,
\begin{equation}
  |u(x,y)|_{H_x^{\mu,\lambda}(\mathbb{R}^2)}^2:=\int_{\mathbb{R}^2}(\lambda^2+\omega_1^2)^\mu|\hat{u}|^2d\omega,\nonumber
\end{equation}
\begin{equation}
  |u(x,y)|_{H_y^{\mu,\lambda}(\mathbb{R}^2)}^2:=\int_{\mathbb{R}^2}(\lambda^2+\omega_2^2)^\mu|\hat{u}|^2d\omega,\nonumber
\end{equation}
\begin{equation}
  |u(x,y)|_{H^{\mu,\lambda}(\mathbb{R}^2)}^2:=\int_{\mathbb{R}^2}(\lambda^2+|\omega|^2)^\mu|\hat{u}|^2d\omega,\nonumber
\end{equation}
where $\omega_1$ and $\omega_2$ are two components of $\omega$, and $\hat{u}$ is the Fourier transform of $u(x,y)$. By Plancherel's theorem, it can be proved that the spaces $J_{L,x}^{\mu,\lambda}$ and $H_x^{\mu,\lambda}$ are equal with equivalent semi-norms and norms while the spaces $J_{L,y}^{\mu,\lambda}$ and $H_y^{\mu,\lambda}$ are also equal with equivalent semi-norms and norms.

Noticing that $|\omega|^2=\omega_1^2+\omega_2^2$, together with (\ref{equation:2.3.1.+}), we have
\[(\lambda^2+|\omega|^2)^\mu\thicksim (\lambda^2+\omega_1^2)^\mu+(\lambda^2+\omega_2^2)^\mu.\]
Therefore,
\begin{equation}\label{Hnorm}
  |u(x,y)|_{H^{\mu,\lambda}(\mathbb{R}^2)}^2\thicksim |u(x,y)|_{H_x^{\mu,\lambda}(\mathbb{R}^2)}^2+|u(x,y)|_{H_y^{\mu,\lambda}(\mathbb{R}^2)}^2,
\end{equation}
which gives an equivalent form of the semi-norm $H^{\mu,\lambda}(\mathbb{R}^2)$, being useful in the error estimate. For any $\mu>0$, the Sobolev space $H^\mu(\mathbb{R}^2)$ is defined
with the semi-norm
\begin{equation}
  |u|_{H^\mu(\mathbb{R}^2)}^2:=\int_{\mathbb{R}^2}|\omega|^{2\mu}|\hat{u}|^2d\omega,
\end{equation}
and the norm
\begin{equation}
  \|u\|_{H^\mu(\mathbb{R}^2)}^2:=\int_{\mathbb{R}^2}(1+|\omega|^{2\mu})|\hat{u}|^2d\omega.
\end{equation}
In the following, we use $H^\mu(\Omega)$ to denote the space of functions on $\Omega$ that admit extensions to $H^\mu(\mathbb{R}^2)$, equipped with the quotient norm $\|u\|_{H^\mu(\Omega)}:= \inf\limits_{\tilde{u}}\|\tilde{u}\|_{H^\mu(\mathbb{R}^2)}$,
where the infimum extends over all possible $\tilde{u}\in H^\mu(\mathbb{R}^2)$ such that $\tilde{u}=u$ on $\Omega$ (in the sense of distributions).

Taking notice of (\ref{equation:2.3.1.+}) again, it holds that
\[1+|\omega|^{2\mu}\thicksim (1+|\omega|^2)^\mu \thicksim (\lambda^2+|\omega|^2)^\mu \thicksim 1+(\lambda^2+|\omega|^2)^\mu.\]
Therefore,
\begin{equation}
  \|u\|_{H^\mu(\mathbb{R}^2)}\thicksim |u|_{H^{\mu,\lambda}(\mathbb{R}^2)} \thicksim \|u\|_{H^{\mu,\lambda}(\mathbb{R}^2)}.
\end{equation}
Then we have the following lemmas from \cite{Deng4} with extension of the space $H_0^{\mu}(\Omega)$ to two dimension.
\begin{lemma}
  Let $0<\mu_1<\mu_2$, and $\mu_1, \mu_2\neq n-\frac{1}{2}\,(n\in\mathbb{N}^+)$. If $u\in H_0^{\mu_2}(\Omega)$, then
  \begin{equation}
    \|u\|_{L^2(\Omega)}\lesssim |u|_{H_0^{\mu_2}(\Omega)}\quad \textrm{and} \quad
    |u|_{H^{\mu_1}(\Omega)}\lesssim |u|_{H^{\mu_2}(\Omega)}.  \nonumber
  \end{equation}
\end{lemma}

\begin{lemma}
  If $u\in H_0^\mu(\Omega)$ with $\mu\in(0,1), \mu\neq\frac{1}{2},$ for all $\lambda\geq 0$, then
  \begin{equation}
    |u|_{H^\mu(\Omega)}^2\sim\|u\|_{H^\mu(\Omega)}^2\sim|u|_{H^{\mu,\lambda}(\Omega)}^2\sim\|u\|_{H^{\mu,\lambda}(\Omega)}^2;
  \end{equation}
  and if $0<\mu_1<\mu$ and $\mu_1\neq\frac{1}{2}$, it holds that
  \begin{equation}
    |u|_{H^{\mu_1,\lambda}(\Omega)}^2\lesssim |u|_{H^{\mu,\lambda}(\Omega)}^2.
  \end{equation}
\end{lemma}

\begin{lemma}\label{lemma:2.2.11}
Let $0<\alpha<1$, and $\lambda>0$. If $f(x,y)\in H_0^\alpha(\Omega)$, then
\begin{equation}\label{equation:2.2.16}
\begin{split}
\left(\,_{a}D_{x}^{\alpha/2,\lambda}f(x,y),\,_{x}D_{b}^{\alpha/2,\lambda}f(x,y)\right)_\Omega
&\geq\cos(\pi\alpha/2)\cdot\|\,_{a}D_{x}^{\alpha/2,\lambda}f(x,y)\|_{L^2(\Omega)}^2
\\&\sim\cos(\pi\alpha/2)\cdot\|\,_{x}D_{b}^{\alpha/2,\lambda}f(x,y)\|_{L^2(\Omega)}^2
\\&\sim\cos(\pi\alpha/2)\cdot\|f(x,y)\|_{H_x^{\alpha/2,\lambda}(\Omega)}^2,
\end{split}
\end{equation}
and
\begin{equation}\label{equation:2.2.17}
\begin{split}
\left(\,_{c}D_{y}^{\alpha/2,\lambda}f(x,y),\,_{y}D_{d}^{\alpha/2,\lambda}f(x,y)\right)
&\geq\cos(\pi\alpha/2)\cdot\|\,_{c}D_{y}^{\alpha/2,\lambda}f(x,y)\|_{L^2(\Omega)}^2
\\&\sim\cos(\pi\alpha/2)\cdot\|\,_{y}D_{d}^{\alpha/2,\lambda}f(x,y)\|_{L^2(\Omega)}^2
\\&\sim\cos(\pi\alpha/2)\cdot\|f(x,y)\|_{H_y^{\alpha/2,\lambda}(\Omega)}^2.
\end{split}
\end{equation}
\end{lemma}

\begin{remark}
Lemma \ref{lemma:2.2.11} shows the coercivity of the tempered fractional operator since $\cos(\pi\alpha/2)>0$ when $0<\alpha<1$. If $1<\alpha<2$, the tempered fractional operator is also coercive; for the details of the proof, see \cite{Deng4}.
\end{remark}

%

\subsection{Notations for DG methods}\label{subsection:2.2}
Denote $\Omega_h$ as a conforming subdivision of $\Omega\subset\mathbb{R}^2$ with non-overlapping triangles. For an element $T\in\Omega_h$, $h_T$ denotes its diameter, and $h=\max_{T\in\Omega_h}h_T$. The family of meshes $\Omega_h$ is assumed to be shape-regular, i.e., there exists a constant $c$ such that for all $T\in\Omega_h$, $\frac{h_T}{\rho_T}\leq c$, where $\rho_T$ denotes the radius of the largest inscribed ball in $T$. Denote $\Gamma$ as the union of the boundaries of the elements $T$ of $\Omega_h$, and $\Gamma_i$ the set of interior faces of $\Omega_h$, i.e., the set of faces that are not included in the boundary $\partial\Omega$. Let $\Gamma_b$ be the faces that are included in $\partial\Omega$. Then $\Gamma=\Gamma_i\cup\Gamma_b$.
Associated with the mesh $\Omega_h$, we define the broken Sobolev spaces
\[H^1(\Omega_h):=\{v:\Omega\rightarrow \mathbb{R}\big|\ \:\nabla v|_T\in[L^2(T)]^2,\ \forall T\in\Omega_h\},\]
and
\[H^{s,\lambda}(\Omega_h):=\{v:\Omega\rightarrow \mathbb{R}\big|\ \:v|_T\in H^{s,\lambda}(T),\ \forall T\in \Omega_h\},\]
equipped with the broken Sobolev norm
\begin{equation}\label{Enorm}
\|v\|_{H^{s,\lambda}(\Omega_h)}:=\left(\sum_{T\in\Omega_h}\|v\|_{H^{s,\lambda}(T)}^2\right)^{1/2},
\end{equation}
where
\[\|u\|^2_{H^{s,\lambda}(T)}:=\int_T |_aD_x^{s,\lambda}u(x,y)|^2+|_cD_y^{s,\lambda}u(x,y)|^2 dxdy.\]
The $H^{s,\lambda}(T)$ norm is defined through (\ref{Hnorm}). And it also can be defined by the right derivatives.

We define the DG finite element space as follows:
\begin{equation} \label{PolySpace}
V_h=\{v:v\in L^2(\Omega)\big|\ v|_T\in P_N(T),\ \forall T\in\Omega_h\},
\end{equation}
where $P_N(T)$ denotes the set of polynomials of degree less than or equal to $N$.
The global solution can be approximated as
\begin{equation}
u_h(\mathbf{x},t)=\bigoplus_{T\in\Omega_h}u_h(\mathbf{x},t)|_T\in V_h,
\end{equation}
and the local solution $u(\mathbf{x},t)$ can be expressed by
\begin{equation}
u_h(\mathbf{x},t)|_T=\sum_{j=1}^{N_p}u_h(\mathbf{x}_j,t)l_j(\mathbf{x}) \qquad \mathbf{x}\in T,
\end{equation}
where $l_j(\mathbf{x})$ denotes the two-dimensional multivariate Lagrange interpolation basis function, and
$N_p=(N+1)(N+2)/2$ is degree of freedom of one element.

Next, we introduce some notations to manipulate numerical fluxes.
If two elements $T_e^1$ and $T_e^2$ are neighbors and share one common side $e$, there are two traces of the function $v$ along $e$. We assume that the normal vector $\mathbf{n}_e$ is oriented from $T_e^1$ to $T_e^2$, and denote
\[\{v\}=\frac{1}{2}(v|_{T_e^1})+\frac{1}{2}(v|_{T_e^2}), \qquad [v]=(v|_{T_e^1})-(v|_{T_e^2}) \quad \forall e=\partial T_e^1\cap\partial T_e^2.\]
The definition of jump and average to sides that belong to the boundary $\partial\Omega$ is:
\[\{v\}=[v]=(v|_{T_e^1})\quad \forall e=\partial T_e^1\cap\partial\Omega.\]

We also introduce two bilinear forms $J_0,\,J_1:H^{s,\lambda}(\Omega_h)\times H^{s,\lambda}(\Omega_h)\rightarrow \mathbb{R}$ that penalize the jump of the function values and the jump of the normal derivative values:
\[J_0(v,w)=\sum_{e\in\Gamma}\int_e\:[v][w],\qquad\qquad
J_1(v,w)=\sum_{e\in\Gamma_i}\int_e\:[\nabla v\cdot\mathbf{n}][\nabla w\cdot\mathbf{n}].\]

\begin{remark}
Reference \cite{Erik1} presents numerical analysis for the advection-reaction equation in the case of continuous, piecewise affine approximations (using the continuous interior penalty method (CIP)) and discontinuous, piecewise affine or constant approximations (DG). In both of the two cases, stability is obtained by using interior penalty; for CIP, the gradient jumps over element faces are penalized, and for DG, the jumps of the solution.
\end{remark}

\subsection{Variational formulation}\label{subsection:2.3}
From the definition of tempered fractional derivatives, we rewrite Eq (\ref{equation:1.1}) as a more clear form
\begin{equation}\label{equation:3.1}
\left\{ \begin{array}{ll}
\frac{\partial}{\partial t}u+\mathbf{b}\cdot\nabla u+
\kappa_1\kappa_\alpha({}_aD_x^{\alpha,\lambda}+{}_xD_b^{\alpha,\lambda}-2\lambda^\alpha)u\\ \qquad\qquad\qquad +
\kappa_2\kappa_\beta({}_cD_y^{\beta,\lambda}+{}_yD_d^{\beta,\lambda}-2\lambda^\beta)u\,=\,f & \Omega\times J \\
u(\mathbf{x},0)\,=\,u_0(\mathbf{x}) & \Omega\\
u(\mathbf{x},t)\,=\,0 &\mathbb{R}^2\backslash\Omega\times J.
\end{array} \right.
\end{equation}
Here, we mainly discuss the case $\alpha,\beta\in(0,1)$.
Then we formulate the weak variational formulation. For $f\in H^{-s,\lambda}(\Omega)$, find $u\in V:= H^{s,\lambda}_0(\Omega)\bigcap H^1(\Omega)$, where $s=\max\{\alpha/2,\beta/2\}$, such that $\forall v\in V$,
\begin{equation}\label{equation:3.2}
\begin{split}
\left(\frac{\partial}{\partial t}u,v\right)+(\mathbf{b}\cdot\nabla u,v)
+\kappa_1\kappa_\alpha\left(({}_aD_x^{\alpha,\lambda}+{}_xD_b^{\alpha,\lambda}-2\lambda^\alpha)u,v\right)\\
+\kappa_2\kappa_\beta\left(({}_cD_y^{\beta,\lambda}+{}_yD_d^{\beta,\lambda}-2\lambda^\beta)u,v\right)=(f,v).
\end{split}
\end{equation}

Denote two bilinear forms as $a_x^\alpha(u,v)=({}_aD_x^{\frac{\alpha}{2},\lambda}u,{}_xD_b^{\frac{\alpha}{2},\lambda}v)+
({}_aD_x^{\frac{\alpha}{2},\lambda}v,{}_xD_b^{\frac{\alpha}{2},\lambda}u),$ and $a_y^\beta(u,v)=({}_cD_y^{\frac{\beta}{2},\lambda}u,{}_yD_d^{\frac{\beta}{2},\lambda}v)+
({}_cD_y^{\frac{\beta}{2},\lambda}v,{}_yD_d^{\frac{\beta}{2},\lambda}u),$ and
 define
\begin{equation}\label{bilinearA}
a(u,v)=\kappa_1\kappa_\alpha a_x^\alpha(u,v)+\kappa_2\kappa_\beta a_y^\beta(u,v),
\end{equation}
and
\begin{equation*}
  b(u,v)=(\mathbf{b}\cdot\nabla u,v).
\end{equation*}
Then the variational formulation becomes
\begin{equation}\label{equation:3.3}
\left(\frac{\partial}{\partial t}u,v\right)+b(u,v)+a(u,v)-\kappa(u,v)=(f,v),
\end{equation}
where $\kappa=2\lambda^\alpha\kappa_1\kappa_\alpha+2\lambda^\beta\kappa_2\kappa_\beta>0$.
So the DG bilinear form should be defined as
\begin{equation*}
a_h(u,v)=\kappa_1\kappa_\alpha a_x^\alpha(u,v)+\kappa_2\kappa_\beta a_y^\beta(u,v)+J_0(u,v)  \quad \forall u,v\in H^{s,\lambda}(\Omega_h),
\end{equation*}
and
\begin{equation*}
b_h(u,v)=-(\mathbf{b}u,\nabla v)+\sum_{e\in\Gamma}\int_e \mathbf{b}\cdot \mathbf{n_e}\hat{u}[v]  \quad \forall u,v\in H^{s,\lambda}(\Omega_h),
\end{equation*}
where $\hat{u}$ corresponding to the element boundary terms from integration by parts is the so-called numerical flux. It is a single valued function defined on the faces and should be designed based on different guiding principles for different PDEs to guarantee stability and optimal order of convergence. Here we choose an upwind flux. Denote the upwind value of a function $v$ by $v^\mathrm{up}$. We recall that $\mathbf{n_e}$ is a unit normal vector pointing from $T_e^1$ to $T_e^2$:
\begin{equation}
v^\mathrm{up}=\left\{\begin{split}
v|_{T_e^1} \qquad \mathrm{if}\: \mathbf{b}\cdot\mathbf{n_e}\geq0~\\
v|_{T_e^2} \qquad \mathrm{if}\: \mathbf{b}\cdot\mathbf{n_e}<0.
\end{split}\right.
\end{equation}
Then the DG variational formulation can be stated as: find $u\in H^{s,\lambda}(\Omega_h)\bigcap H^1(\Omega_h)$, such that $\forall v\in H^{s,\lambda}(\Omega_h)\bigcap H^1(\Omega_h)$ there exists
\begin{equation}\label{eqDGform}
  \left(\frac{\partial}{\partial t}u,v\right)+b_h(u,v)+a_h(u,v)-\kappa(u,v)=(f,v).
\end{equation}
We discrete the time derivative with backward Euler. Let $N_T$ be a positive integer and $\Delta t=T/N_T$ denote the time step. We also use the notations:
\[ t^n=n\cdot\Delta t,\quad u^n(\mathbf{x})=u(\mathbf{x},t_n) \quad\forall n\geq0.\]
Define the orthogonal projection operators, $P_h:L_2(\Omega_h)\,\rightarrow\,V_h$, i.e., for each element $T$,
\begin{equation}\label{equation:3.5.+}
  (P_hu-u,v)_T=0 \qquad \forall v\in P_N(T).
\end{equation}
The fully discrete DG scheme is as follows:
find $u_h^{n+1}\in V_h$, such that $\forall v\in V_h$, there exists
\begin{equation}\label{equation:3.5}
\left(\frac{u_h^{n+1}-u_h^n}{\Delta t},v\right)+b_h(u_h^{n+1},v)+a_h(u_h^{n+1},v)-\kappa(u_h^{n+1},v)=(f^{n+1},v)
\end{equation}
with known $u_h^n$, and if $n=0,~u_h^0=P_hu_0(x)$.

\subsection{Stability analysis and error estimates}\label{subsection:2.4}
\begin{lemma}[\rm Discrete Gr\"onwall inequality \cite{Riv1}]\label{lemma:4.1}
Let $\Delta t,B,C>0$ and $(a_n)_n$, $(b_n)_n$, $(c_n)_n$ be sequences of nonnegative numbers satisfying
\[a_n+\Delta t\sum_{i=0}^nb_i\leq B+C\Delta t\sum_{i=0}^na_i+\Delta t\sum_{i=0}^nc_i \quad \forall n\geq0. \]
Then, if $C\Delta t<1$,
\[a_n+\Delta t\sum_{i=0}^nb_i\leq e^{C(n+1)\Delta t}(B+\Delta t\sum_{i=0}^nc_i) \quad \forall n\geq0.\]
\end{lemma}

\begin{lemma}\label{lemma:4.2}
  For any function $f(x),g(x,t)\in L^2(\Omega)$, if $|f(x)|\leq\int_{t_1}^{t_2}g(x,t)dt~\forall x\in\Omega$, then
  \[\|f(x)\|_{L^2(\Omega)}^2\leq(t_2-t_1)\int_{t_1}^{t_2}\|g(x,t)\|_{L^2(\Omega)}^2dt.\]
\end{lemma}
\begin{proof}
  This inequality can be easily obtained by using H\"{o}ld's inequality. Here we omit it.
\end{proof}
In the following, $C$ denotes a generic constant independent of $h$ and $\Delta t$, which takes different values in different occurrences.
First, we prove the coercivity and continuity of the bilinear form $a_h(u,v)$.
Define the energy norm on $V_h$ as
\begin{equation}\label{equation:4.1.1}
\|v\|_{E(\Omega_h)}^2=\|v\|_{H_x^{\alpha/2,\lambda}(\Omega_h)}^2+\|v\|_{H_y^{\beta/2,\lambda}(\Omega_h)}^2+\sum_{e\in\Gamma}\int_e[v]^2,
\end{equation}
where
\[\|v\|_{H_x^{\alpha/2,\lambda}(\Omega_h)}^2=\sum\limits_{T\in\Omega_h}\|_aD_x^{\alpha/2,\lambda}v\|_{L^2(T)}^2,~~
\|v\|_{H_y^{\beta/2,\lambda}(\Omega_h)}^2=\sum\limits_{T\in\Omega_h}\|_cD_y^{\beta/2,\lambda}v\|_{L^2(T)}^2.\]
Then from \eqref{bilinearA} and Lemma \ref{lemma:2.2.11}, we have
\begin{equation}\label{equation:4.1}
\begin{split}
a_h(v,v)&=2\kappa_1\kappa_\alpha({}_aD_x^{\frac{\alpha}{2},\lambda}v,{}_xD_b^{\frac{\alpha}{2},\lambda}v)+
2\kappa_2\kappa_\beta({}_cD_y^{\frac{\beta}{2},\lambda}u,{}_yD_d^{\frac{\beta}{2},\lambda}v)+\sum_{e\in\Gamma}\int_e[v]^2\\
&\geq\kappa_1\|v\|_{H_x^{\alpha/2,\lambda}(\Omega_h)}^2+\kappa_2\|v\|_{H_y^{\beta/2,\lambda}(\Omega_h)}^2+\sum_{e\in\Gamma}\int_e[v]^2\\
&\geq\gamma\|v\|_{E(\Omega_h)}^2,
\end{split}
\end{equation}
where $\gamma=\min\{\kappa_1,\kappa_2,1\}$.

On the other hand,
from the definition of $a_h(u,v)$, and Cauchy-Schwarz's inequality, we have
\begin{equation}\label{equation:3.6}
\begin{split}
  |a_h(u,v)|
  \leq&~ |\kappa_1\kappa_\alpha|\ \left(|({}_aD_x^{\frac{\alpha}{2},\lambda}u,{}_xD_b^{\frac{\alpha}{2},\lambda}v)|+
  |({}_aD_x^{\frac{\alpha}{2},\lambda}v,{}_xD_b^{\frac{\alpha}{2},\lambda}u)|\right)\\
  &+|\kappa_2\kappa_\beta|\ \left(|({}_cD_y^{\frac{\beta}{2},\lambda}u,{}_yD_d^{\frac{\beta}{2},\lambda}v)|+
  |({}_cD_y^{\frac{\beta}{2},\lambda}v,{}_yD_d^{\frac{\beta}{2},\lambda}u)|\right)+
  \sum_{e\in\Gamma}\int_e[u][v]\\
  \leq&~C_{\alpha}\|u\|_{H_x^{\alpha/2,\lambda}(\Omega_h)}\cdot\|v\|_{H_x^{\alpha/2,\lambda}(\Omega_h)}\\
  &+C_{\beta}\|u\|_{H_y^{\beta/2,\lambda}(\Omega_h)}\cdot\|v\|_{H_y^{\beta/2,\lambda}(\Omega_h)}+\sum_{e\in\Gamma}\int_e[u][v]\\
  \leq&~C_{\alpha,\beta}\|u\|_{E(\Omega_h)}\cdot\|v\|_{E(\Omega_h)}.
  \end{split}
\end{equation}
Next we deal with another bilinear form $b_h(u,v)$.
By an upwind flux,
\begin{eqnarray}
b_h(u,v)=-(\mathbf{b}u,\nabla v)+\sum_{e\in\Gamma}\int_e \mathbf{b}\cdot\mathbf{n_e}u^{\mathrm{up}}[v].
\end{eqnarray}
Then
\begin{equation}\label{equation:4.2}
\begin{split}
  b_h(v,v)&=-(\mathbf{b}v,\nabla v)+\sum_{e\in\Gamma}\int_e \mathbf{b}\cdot\mathbf{n_e}v^{\mathrm{up}}[v]
  \\&=\sum_{e\in\Gamma}\int_e\mathbf{b}\cdot\mathbf{n_e}(v^{\mathrm{up}}[v]-\frac{1}{2}[v^2])
  \\&=\sum_{e\in\Gamma}\int_e\mathbf{b}\cdot\mathbf{n_e}(v^{\mathrm{up}}[v]-\{v\}[v])
  \\&=\frac{1}{2}\sum_{e\in\Gamma}\int_e|\mathbf{b}\cdot\mathbf{n_e}|[v]^2\geq0.
  \end{split}
\end{equation}

Now we examine the stability property of the scheme \eqref{equation:3.5}. Taking $u_h^0=P_hu_0$ leads to the following results.
\begin{theorem}
  For absorbing boundary conditions, the fully discrete scheme \eqref{equation:3.5} is unconditionally stable, and there exists a positive constant $C$ independent of $h$, such that for all $m>0$,
  \begin{equation}
    \|u_h^m\|_{L^2(\Omega)}^2+2\Delta t\gamma\sum_{n=1}^m\|u_h^n\|_{E(\Omega_h)}^2\leq C\left(\|u_h^0\|_{L^2(\Omega)}^2+\Delta t\sum_{n=1}^m\|f^n\|_{L^2(\Omega)}^2\right).
  \end{equation}
\end{theorem}
\begin{proof}
  Taking $v=u_h^{n+1}$ in (\ref{equation:3.5}), using (\ref{equation:4.1}), (\ref{equation:4.2}), and Cauchy-Schwarz's inequality,
  \[ \frac{1}{\Delta t}(u_h^{n+1}-u_h^n,u_h^{n+1})+\gamma\|u_h^{n+1}\|_{E(\Omega_h)}^2
  -\kappa\|u_h^{n+1}\|_{L^2(\Omega)}^2\leq\|f^{n+1}\|_{L^2(\Omega)}\cdot\|u_h^{n+1}\|_{L^2(\Omega)}.\]
  Next, we observe that:
  \[(x^2-y^2)\leq(x^2-y^2+(x-y)^2)=2(x-y)x \quad  \forall x,y\in \mathbb{R}. \]
  Therefore, by using Young's inequality,
  \begin{equation}\label{equation:4.3}
    \frac{1}{2\Delta t}(\|u_h^{n+1}\|_{L^2(\Omega)}^2-\|u_h^n\|_{L^2(\Omega)}^2)+\gamma\|u_h^{n+1}\|_{E(\Omega_h)}^2
    \leq\frac{1}{2}\|f^{n+1}\|_{L^2(\Omega)}^2+\left(\kappa+\frac{1}{2}\right)\|u_h^{n+1}\|_{L^2(\Omega)}^2.
  \end{equation}
Multiplying (\ref{equation:4.3}) by $2\Delta t$ and summing from $n=0$ to $n=m-1$ lead to
  \begin{equation*}
  \begin{split}
    \|u_h^m\|_{L^2(\Omega)}^2-\|u_h^0\|_{L^2(\Omega)}^2&+2\Delta t\gamma\sum_{n=1}^m\|u_h^n\|_{E(\Omega_h)}^2\\
    &\leq \Delta t\sum_{n=1}^m\|f^n\|_{L^2(\Omega)}^2+(2\kappa+1)\Delta t\sum_{n+1}^m\|u_h^n\|_{L^2(\Omega)}^2.
  \end{split}
  \end{equation*}
  Then the desired result is obtained by using Lemma \ref{lemma:4.1}.
\end{proof}

Assuming that the solution of Eq. (\ref{equation:1.1}) is sufficiently regular, we have the following error estimates, showing that the numerical $L^2$ error is $\mathcal{O}(h^{k+\frac{1}{2}}+\Delta t)$.
\begin{theorem}
  Let $u^n$ be the exact solution of \eqref{eqDGform}, $u_h^n$ the numerical solution of the fully discrete scheme \eqref{equation:3.5}. Then
\begin{equation}
  \|u^{N_T}-u_h^{N_T}\|_{L^2(\Omega)}^2\leq C\left(h^{2k+2}\int_0^T\|u_t\|_{L^2(\Omega)}^2dt+\Delta t^2\int_0^T\|u_{tt}\|_{L^2(\Omega)}^2dt+h^{2k+1}\right).
\end{equation}
\end{theorem}
\begin{proof}
  As usual, we denote the error $e^n=u^n-u_h^n$ by two parts $\rho=u^n-P_hu^n$ and $\theta=u_h^n-P_hu^n$.
  From (\ref{eqDGform}) and (\ref{equation:3.5}), we have
  \begin{equation}
  \begin{split}
    \left(\frac{\partial}{\partial t}u^{n+1}-\frac{u^{n+1}-u^n}{\Delta t},v\right)
    &+\left(\frac{e^{n+1}-e^n}{\Delta t},v\right)\\
    +~b_h(e^{n+1},v)
    &+a_h(e^{n+1},v)-\kappa(e^{n+1},v)=0.\nonumber
    \end{split}
  \end{equation}
  Noting that $e^n=\rho^n-\theta^n$, we get
  \begin{equation}
  \begin{split}
    \left(\frac{\theta^{n+1}-\theta^n}{\Delta t},v\right)
    &+b_h(\theta^{n+1},v)+a_h(\theta^{n+1},v)-\kappa(\theta^{n+1},v)\\
    =&\left(\frac{\rho^{n+1}-\rho^n}{\Delta t},v\right)+b_h(\rho^{n+1},v)+a_h(\rho^{n+1},v)\\
    &-\kappa(\rho^{n+1},v)+\left(\frac{\partial}{\partial t}u^{n+1}-\frac{u^{n+1}-u^n}{\Delta t},v\right).\nonumber
    \end{split}
  \end{equation}
  Taking $v=\theta^{n+1}$, similar to the proof of stability, we obtain
  \begin{equation}\label{equation:4.2.1}
    \frac{1}{2\Delta t}(\|\theta^{n+1}\|_{L^2(\Omega)}^2-\|\theta^n\|_{L^2(\Omega)}^2)+\gamma\|\theta^{n+1}\|_{E(\Omega_h)}^2-\kappa\|\theta^{n+1}\|_{L^2(\Omega)}^2
    \leq \sum_{i=1}^5|T_i|,
  \end{equation}
  where $T_1=(\frac{\rho^{n+1}-\rho^n}{\Delta t},\theta^{n+1})$, $T_2=b_h(\rho^{n+1},\theta^{n+1})$,
   $T_3=a_h(\rho^{n+1},\theta^{n+1})$, $T_4=\kappa(\rho^{n+1},\theta^{n+1})$, and
   $T_5=(\frac{\partial}{\partial t}u^{n+1}-\frac{u^{n+1}-u^n}{\Delta t},\theta^{n+1})$.

Since
   \[\rho^{n+1}-\rho^n=\int_{t_n}^{t_{n+1}}\rho_t dt,\]
   by Lemma \ref{lemma:4.2}, we have
   \[\|\rho^{n+1}-\rho^n\|_{L^2(\Omega)}^2\leq\Delta t\int_{t_n}^{t_{n+1}}\|\rho_t\|_{L^2(\Omega)}^2dt.\]
   Hence, with H\"{o}ld's, Young's inequalities, we obtain
   \begin{equation}
   \begin{split}
     |T_1|&\leq\left\|\frac{\rho^{n+1}-\rho^n}{\Delta t}\right\|_{L^2(\Omega)} \cdot \|\theta^{n+1}\|_{L^2(\Omega)}
     \\&\leq\frac{3}{4\epsilon_1}\cdot\frac{1}{\Delta t}\int_{t_n}^{t_{n+1}}\|\rho_t\|_{L^2(\Omega)}^2dt
     +\frac{\epsilon_1}{3}\|\theta^{n+1}\|_{L^2(\Omega)}^2.\nonumber
     \end{split}
   \end{equation}
From the definition of the projection $P_h$ (\ref{equation:3.5.+}) and trace inequalities, we have
   \begin{equation*}
   \begin{split}
     |T_2|\leq&~|(\mathbf{b}\rho^{n+1},\nabla\theta^{n+1})|+
     \left|\sum_{e\in\Gamma}\int_e\mathbf{b}\cdot\mathbf{n}\hat{\rho}^{n+1}[\theta^{n+1}]\right|\\
     \leq&~\|\mathbf{b}\|_\infty \sum_{e\in\Gamma}\|\hat{\rho}^{n+1}\|_{L^2(e)}\cdot\|[\theta^{n+1}]\|_{L^2(e)}\\
     \leq&~\|\mathbf{b}\|_\infty^2\frac{1}{4\epsilon_2} \sum_{e\in\Gamma}\|\hat{\rho}^{n+1}\|_{L^2(e)}^2+\epsilon_2\sum_{e\in\Gamma}\|[\theta^{n+1}]\|_{L^2(e)}^2\\
     \leq&~\frac{\|\mathbf{b}\|_\infty^2}{4\epsilon_2}h^{2k+1}+\epsilon_2\sum_{e\in\Gamma}\|[\theta^{n+1}]\|_{L^2(e)}^2.
     \end{split}
   \end{equation*}
   From the continuity of $a(u,v)$ (\ref{equation:3.6}), we obtain
   \begin{equation*}
   \begin{split}
     |T_3|\leq&~C\|\rho^{n+1}\|_{E(\Omega_h)}\cdot\|\theta^{n+1}\|_{E(\Omega_h)}
     \\\leq&~\frac{C^2}{4\epsilon_3}\|\rho^{n+1}\|_{E(\Omega_h)}^2+\epsilon_3\|\theta^{n+1}\|_{E(\Omega_h)}^2
     \\\leq&~Ch^{2k+1}+\epsilon_3\|\theta^{n+1}\|_{E(\Omega_h)}^2.
     \end{split}
   \end{equation*}
   In the last inequality, using embedding theorem and trace theorem, three terms in the energy norm $\|\rho^{n+1}\|_{E(\Omega_h)}$ both can be bounded by $|\rho^{n+1}|_{H^{1/2}(\Omega_h)}$. Then, for a continuous interpolation function $\Pi u$ of $u$,
    \[|u-\Pi u|_{H^{1/2}(\Omega_h)}\leq Ch^{k+1/2}|u|_{H^{k+1}(\Omega)},\]
    which implies $\|\rho^{n+1}\|_{E(\Omega_h)}^2\leq Ch^{2k+1}$.

   Similarly, for the fourth term, we have
   \begin{equation*}
   \begin{split}
     |T_4|&\leq\kappa\|\rho^{n+1}\|_{L^2(\Omega)}\cdot\|\theta^{n+1}\|_{L^2(\Omega)}
     \\&\leq\kappa^2\frac{3}{4\epsilon_1}\|\rho^{n+1}\|_{L^2(\Omega)}^2+\frac{\epsilon_1}{3}\|\theta^{n+1}\|_{L^2(\Omega)}^2
     \\&\leq\frac{3\kappa^2}{4\epsilon_1}h^{2k+2}+\frac{\epsilon_1}{3}\|\theta^{n+1}\|_{L^2(\Omega)}^2.
     \end{split}
   \end{equation*}
The Taylor expansion with integral remainder has the form
   \[u^n=u^{n+1}-\Delta tu_t^{n+1}+\int_{t_{n+1}}^{t_n}(t_n-t)u_{tt}(t)dt.\]
   Thus,
   \begin{equation*}
   \begin{split}
     \left|\frac{\partial}{\partial t}u^{n+1}-\frac{u^{n+1}-u^n}{\Delta t}\right|
     &\leq\frac{1}{\Delta t}\int_{t_n}^{t_{n+1}}(t_n-t)u_{tt}(t)dt
     \\&\leq\int_{t_n}^{t_{n+1}}|u_{tt}(t)|dt.
     \end{split}
   \end{equation*}
   Then we get
   \begin{equation*}
   \begin{split}
     |T_5|&\leq\left\|\frac{\partial}{\partial t}u^{n+1}-\frac{u^{n+1}-u^n}{\Delta t}\right\|_{L^2(\Omega)}\cdot\|\theta^{n+1}\|_{L^2(\Omega)}
     \\&\leq\frac{3}{4\epsilon_1}\left\|\frac{\partial}{\partial t}u^{n+1}-\frac{u^{n+1}-u^n}{\Delta t}\right\|_{L^2(\Omega)}^2+
     \frac{\epsilon_1}{3}\|\theta^{n+1}\|_{L^2(\Omega)}^2
     \\&\leq\frac{3\Delta t}{4\epsilon_1}\int_{t_n}^{t_{n+1}}\|u_{tt}(t)\|_{L^2(\Omega)}^2dt
     +\frac{\epsilon_1}{3}\|\theta^{n+1}\|_{L^2(\Omega)}^2.
     \end{split}
   \end{equation*}
   Substituting $T_i,\:i=1,\cdots,5$ into (\ref{equation:4.2.1}), we have
   \begin{equation}\label{equation:4.2.2}
   \begin{split}
     \frac{1}{2\Delta t}(\|&\theta^{n+1}\|_{L^2(\Omega)}^2-\|\theta^n\|_{L^2(\Omega)}^2)+(\gamma-\epsilon_3)\|\theta^{n+1}\|_{E(\Omega_h)}^2\\
     \leq&~(\kappa+\epsilon_1)\|\theta^{n+1}\|_{L^2(\Omega)}^2+\epsilon_2\sum_{e\in\Gamma}\|[\theta^{n+1}]\|_{L^2(e)}^2\\
     &+\frac{3}{4\epsilon_1}\cdot\frac{1}{\Delta t}\int_{t_n}^{t_{n+1}}\|\rho_t\|_{L^2(\Omega)}^2dt+
     \frac{3\Delta t}{4\epsilon_1}\int_{t_n}^{t_{n+1}}\|u_{tt}(t)\|_{L^2(\Omega)}^2dt+Ch^{2k+1},
     \end{split}
   \end{equation}
   where $\epsilon_2$ and $\epsilon_3$ are chosen as sufficiently small numbers such that $\epsilon_2+\epsilon_3\leq
   \gamma$.\\
   From the definition of energy norm (\ref{equation:4.1.1}), we know \[\|\theta^{n+1}\|_{E(\Omega_h)}^2\geq
    \sum\limits_{e\in\Gamma}\|[\theta^{n+1}]\|_{L^2(e)}^2.\]
   Then \[(\gamma-\epsilon_3)\|\theta^{n+1}\|_{E(\Omega_h)}^2-\epsilon_2\sum_{e\in\Gamma}\|[\theta^{n+1}]\|_{L^2(e)}^2\geq0.\]
   Multiplying (\ref{equation:4.2.2}) by $2\Delta t$, summing over $n$ from $0$ to $N_T-1$, and using the discrete Gr\"{o}nwall inequality
    with $\theta^0=0$, we get
   \begin{equation*}
   \begin{split}
     \|\theta^{N_T}\|_{L^2(\Omega)}^2
     \leq&~\frac{3h^{2k+2}}{2\epsilon_1}\int_0^T\|u_t\|_{L^2(\Omega)}^2dt
     +\frac{3}{\epsilon_1}\Delta t^2\int_0^T\|u_{tt}\|_{L^2(\Omega)}^2dt\\
     &+2\Delta t(k+\epsilon_1)\sum_{n=1}^{N_T}\|\theta^n\|_{L^2(\Omega)}^2+Ch^{2k+1}\\
     \leq&~C\left(h^{2k+2}\int_0^T\|u_t\|_{L^2(\Omega)}^2dt+\Delta t^2\int_0^T\|u_{tt}\|_{L^2(\Omega)}^2dt+h^{2k+1}\right).
     \end{split}
   \end{equation*}
   By the triangle inequality, we obtain the desired result.
\end{proof}

\subsection{Numerical experiment}\label{subsection:2.5}
In this section, we offer the numerical performance of the proposed schemes for two examples to validate the preceding theoretical analysis. We use the backward Euler discretization to solve the method-of-line fractional PDE, i.e., the classical ODE system. We take the time steps $\Delta t$ to be $h^{N+1}$, where $N$ denotes the order of polynomial of finite element space. As to the spatial approximation, we adopt the interpolation bases \cite{Hesthaven1}.

We first introduce the local and global vector and matrix notations,
\[\mathbf{u}_{h,k}=[u_{1,k},u_{2,k},\cdots,u_{Np,k}]^T,\]
\[\mathbf{u}_h=[\mathbf{u}_{h,1},\mathbf{u}_{h,2},\cdots,\mathbf{u}_{h,K}]^T,\]
and $\mathbf{u}_{h}^n$ denotes the value of $\mathbf{u}_h$ at time $t_n$.
Let $f_{j,k}=(f,l_j^k(\mathbf{x}))_{T^k}$. Similarly, denote
\[\mathbf{f}_{h,k}=[f_{1,k},f_{2,k},\cdots,f_{Np,k}]^T,\]
\[\mathbf{f}_h=[\mathbf{f}_{h,1},\mathbf{f}_{h,2},\cdots,\mathbf{f}_{h,K}]^T,\]
and let $\mathbf{F}_h^n$ be the value of $\mathbf{f}_h$ at time $t_n$.
Then, we define the local mass matrix $M^k$ and the local spatial stiffness matrix $S_x^k,S_y^k$ at element $T^k$ as
\[M_{ij}^k=(l_i^k(\mathbf{x}),l_j^k(\mathbf{x}))_{T^k},\]
\[(S_x^k)_{ij}=\left(\frac{\partial l_j^k(\mathbf{x})}{\partial x},l_i^k(\mathbf{x})\right)_{T^k},\]
\[(S_y^k)_{ij}=\left(\frac{\partial l_j^k(\mathbf{x})}{\partial y},l_i^k(\mathbf{x})\right)_{T^k}.\]
It is a little bit complex to build the tempered fractional spatial stiffness matrix,
since tempered fractional operators are nonlocal and we need all the information of the related elements in $x$ direction or $y$ direction when generating any stiffness matrix of an element. We take the method of \cite{Qiu1} and get the global tempered fractional spatial stiffness matrices ${}_lG_x,{}_lG_y,{}_rG_x,{}_rG_y$, where $`l/r$' denote left/right tempered fractional derivative and $`x/y$' denote the $x$ or $y$ direction.

With the above notations, we rewrite the global fully discrete form (\ref{equation:3.5}) as
\begin{equation*}
\begin{split}
M\frac{\mathbf{u}_h^{n+1}-\mathbf{u}_h^n}{\Delta t}
&+b_1S_x\mathbf{u}_h^{n+1}+\kappa_1\kappa_\alpha(_lG_x+{}_rG_x)\mathbf{u}_h^{n+1}
\\&+b_2S_y\mathbf{u}_h^{n+1}+\kappa_2\kappa_\beta(_lG_y+{}_rG_y)\mathbf{u}_h^{n+1}-\kappa M\mathbf{u}_h^{n+1}=\mathbf{F}^{n+1},
\end{split}
\end{equation*}
where $M$, $S_x$, $S_y$ are global mass and stiffness matrices, and their non-zero diagonal blocks are constructed by $M^k$, $S_x^k$, and $S_y^k$ respectively.
\begin{example}\label{exam:2.1}
Consider the problem
\begin{equation} \label{Ex2.19}
\frac{\partial u}{\partial t}+\mathbf{b}\cdot\nabla u -\kappa_1\nabla_x^{\alpha,\lambda}u-\kappa_2\nabla_y^{\beta,\lambda}u=f,
\end{equation}
where $\mathbf{b}=(0.5,0.5),\,\kappa_1=0.1,\, \kappa_2=0.2,\, \lambda=2,\, T=1$ and $\alpha,\beta\in(0,1)$ on the computational domain $\Omega=(0,2)\times(0,2)$. Its exact solution is $u=\e^{-t}x^2(2-x)^2y^2(2-y)^2$ with appropriate initial and boundary conditions.
\end{example}


\begin{example}\label{exam:2.2}
Consider the same problem Eq. (\ref{Ex2.19}), but the parameters are taken as
$\mathbf{b}=(0.5,0.5),\, \kappa_1=0.1, \,\kappa_2=0.2,\, \lambda=0.2,\, T=1$ and $\alpha\,,\beta\in(0,2)$ on the computational domain $\Omega=(0,2)\times(0,2)$. The exact solution is $u=\e^{-t}\sin{\frac{\pi}{2}x}\sin{\frac{\pi}{2}y}$.
\end{example}

For the numerical experiments, in order to validate the stability and the convergence of the preceding scheme, the order of convergence is calculated by
\[
\textrm{order}=\frac{\log(\|u(T)-u_{h_1}(T)\|_{L^2(\Omega)})-\log(\|u(T)-u_{h_2}(T)\|_{L^2(\Omega)})}{\log(h_1)-\log(h_2)}.
\]
Table \ref{table:1} and Table \ref{table:2} list the $L^2$ errors and convergence orders for different parameters $(\alpha,\beta)$ in different DG finite element space $P_N$, where $N$ denotes the degree of polynomial in two variables, and $K$ the total number of triangle elements. It can be seen that the convergence order $N+1/2$ is consistent with the theoretical prediction, even for the case of $(\alpha,\beta)\in(0,2)$.

\begin{table}[!htb]
\centering
\caption{Numerical errors ($L_2$) and orders of convergence on unstructured meshes for Example \ref{exam:2.1}.}
\vspace{0.2cm}
    \begin{tabular}{*{9}{|c}|}
      \hline
      &K & \multicolumn{1}{c|}{68} & \multicolumn{2}{c|}{211} & \multicolumn{2}{c|}{436} & \multicolumn{2}{c|}{702} \\
    \hline
     N &$(\alpha,\beta)$ &  error  &  error & order &  error & order &  error & order   \\ \hline
     \multirow{3}{*}{1} & (0.2,0.2)  &4.46e-2 &1.52e-2 &1.90 &7.60e-3 &1.91 & 5.20e-3 &1.59 \\
                    \cline{2-9}
                    & (0.5,0.5) & 4.23e-2 &1.42e-2 &1.93 &7.00e-3 &1.95 &4.80e-3 &1.58 \\
                    \cline{2-9}
                    & (0.7,0.2) & 4.29e-2 &1.45e-2 &1.92 &7.20e-3 &1.93 &4.90e-3 &1.62\\
                    \hline
      &K & \multicolumn{1}{c|}{68} & \multicolumn{2}{c|}{211} & \multicolumn{2}{c|}{436} & \multicolumn{2}{c|}{702} \\\hline
     \multirow{3}{*}{2} & (0.2,0.2)  &1.01e-2 &2.10e-3 &2.77 &7.75e-4 &2.75 & 3.80e-4 &2.99 \\
                    \cline{2-9}
                    & (0.5,0.5) & 9.70e-3 &2.00e-3 &2.79 &7.19e-4 &2.82 &3.51e-4 &3.01 \\
                    \cline{2-9}
                    & (0.7,0.2) & 9.80e-3 &2.10e-3 &2.72 &7.33e-4 &2.90 &3.57e-4 &3.02\\
                    \hline
      &K & \multicolumn{1}{c|}{68} & \multicolumn{2}{c|}{211} & \multicolumn{2}{c|}{436} & \multicolumn{2}{c|}{702} \\  \hline
     \multirow{3}{*}{3} & (0.2,0.2)  &3.00e-3 &3.94e-4 &3.59 &1.02e-4 &3.72 & 3.76e-5 &4.21 \\
                    \cline{2-9}
                    & (0.5,0.5) & 3.00e-3 &3.87e-4 &3.62 &1.01e-4 &3.71 &3.70e-5 &4.20 \\
                    \cline{2-9}
                    & (0.7,0.2) & 3.00e-3 &3.89e-4 &3.61 &1.01e-4 &3.71 &3.71e-5 &4.21\\
                    \hline
    \end{tabular}
    \label{table:1}
\end{table}

\begin{table}[!htb]
\centering
\caption{Numerical errors ($L_2$) and orders of convergence on unstructured meshes for Example \ref{exam:2.2}.}
\vspace{0.2cm}
    \begin{tabular}{*{9}{|c}|}
      \hline
      &K & \multicolumn{1}{c|}{68} & \multicolumn{2}{c|}{211} & \multicolumn{2}{c|}{436} & \multicolumn{2}{c|}{702} \\
    \hline
     N &$(\alpha,\beta)$ &  error  &  error & order &  error & order &  error & order   \\ \hline
     \multirow{3}{*}{1} & (0.5,0.5)  &3.33e-2 &9.94e-3 &2.13 &4.66e-3 &2.09 &2.86e-3  &2.05 \\
                    \cline{2-9}
                    & (1.5,1.5) & 3.09e-2 &9.30e-3 &2.12 &4.50e-3 &2.00 &2.80e-3 &1.99  \\
                    \cline{2-9}
                    & (1.5,0.5) & 3.20e-2 &9.60e-3 &2.12 &4.50e-3 &2.09 &2.80e-3 &1.99\\
                    \hline
      &K & \multicolumn{1}{c|}{68} & \multicolumn{2}{c|}{211} & \multicolumn{2}{c|}{436} & \multicolumn{2}{c|}{702} \\\hline
     \multirow{3}{*}{2} & (0.5,0.5)  &1.50e-3 &3.09e-4 &2.79 &1.08e-4 &2.90 &5.40e-5  &2.91 \\
                    \cline{2-9}
                    & (1.5,1.5) &9.73e-4 &2.00e-4 &2.74 &7.36e-5 &2.75 & 3.49e-5 &3.13 \\
                    \cline{2-9}
                    & (1.5,0.5) &1.03e-3 &1.90e-4 &2.99 &6.15e-5 &3.11 &2.65e-5  &3.53 \\
                    \hline
      &K & \multicolumn{1}{c|}{68} & \multicolumn{2}{c|}{211} & \multicolumn{2}{c|}{436} & \multicolumn{2}{c|}{702} \\  \hline
     \multirow{3}{*}{3} & (0.5,0.5)  &3.72e-4 &4.75e-5 &3.81 &1.24e-5 &3.70 &4.62e-6 &4.15 \\
                    \cline{2-9}
                    & (1.5,1.5) &3.29e-4 &4.31e-5 &3.59 &1.14e-5 &3.66 &4.22e-6  &4.17 \\
                    \cline{2-9}
                    & (1.5,0.5) &3.51e-4 &4.52e-5 &3.62 &1.18e-5 &3.70 &4.40e-6 &4.14 \\
                    \hline
    \end{tabular}
    \label{table:2}
\end{table}

\section{Adaptive DG algorithm}\label{section:3}
This section focuses on the adaptive DG scheme for the fractional diffusion equations. We derive posteriori error estimates, and design the local error indicators. The numerical experiments are performed to show the performances of the adaptive schemes.

\subsection{Stationary Equation}\label{section:3.1}
First we consider the simple stationary equation:
\begin{equation}\label{equation:3.0}
\left\{ \begin{array}{ll}
{}_{-\infty}D_x^{\alpha}u+{}_xD_{\infty}^{\alpha}u+{}_{-\infty}D_y^{\alpha}u+{}_yD_{\infty}^{\alpha}u=f  \quad \Omega  \\
                \qquad\qquad\qquad\qquad\qquad\qquad\qquad\quad\,\,\,    u=0  \quad  \mathbb{R}^2\backslash\Omega
\end{array}. \right.
\end{equation}

A posteriori error estimators are an essential ingredient of adaptivity, which are computable quantities depending on the computed solution and data that provide information about the quality of approximation and may thus be used to make judicious mesh modifications. The ultimate purpose is to construct the estimator of meshes that will eventually be equivalent to the exact error. The usual method of constructing the estimator deals with error estimation in global norms like the `energy norm' or the `$L^2$ norm', which is the first scheme in the following.
\subsubsection{Scheme 1 -- energy norm \cite{Brenner1}}\label{section:3.1.1}
For $T\in\Omega_h$, we define the operators $D_x^{\alpha}={}_aD_x^{\alpha}+{}_xD_b^{\alpha}$, $D_y^{\alpha}={}_cD_y^{\alpha}+{}_yD_d^{\alpha}$, and the local error estimator $\eta_T$ by
\[\eta_T^2=\eta_{1,T}^2+\eta_{2,T}^2=h_T^{\alpha}\|R\|_{L^2(T)}^2+\|[u_h]\|_{L^2(\partial T)}^2,\]
where $R:=f-(D_x^{\alpha}u_h+D_y^{\alpha}u_h)$. We will prove that $\eta_T$ bounds the exact error by inequalities in both directions, where the constants in these inequalities depend only on properties of the triangulation. The upper estimate shows that $\eta_T$ can be used as a reliable stopping criterion for the algorithm, while the lower estimate suggests that refinement based on $\eta_T$ will be efficient.
\begin{lemma}[upper bound]\label{lemma:3.1}
There exists a constant $C_1$, depending only on the minimum angle of $\Omega_h$ and the ratio of the biggest diam to the smallest diam of the elements, such that
\begin{equation}
  \| u-u_h\|_{E(\Omega_h)}^2 \leq C_1 \underset{T\in{\Omega_h}}{\sum}\eta_T^2,
\end{equation}
where the energy norm $\|v\|_{E(\Omega_h)}$ is defined in \eqref{equation:4.1.1}.
\end{lemma}
\begin{proof}
  The weak form reads as follows: find $u\in V:=H_0^{\alpha/2}(\Omega_h)$, such that
  \begin{equation}
    a(u,v)=(f,v)  \quad \forall v\in V,
  \end{equation}
  where the symmetry bilinear form
\[a(u,v)=a_x^\alpha(u,v)+a_y^\alpha(u,v)+\sum\limits_{e\in\Gamma}\int_e[u][v].\]
The notations $a_x^\alpha(u,v)$ and $a_y^\alpha(u,v)$ are defined in (\ref{bilinearA}).

  Let $u_h\in V_h$ (defined in (\ref{PolySpace})) be the numerical solution, and $e_h:=u-u_h$. Then $e_h$ satisfies the residual equation
  \begin{equation}\label{ResEq}
    a(e_h,v)=(R,v)-\sum\limits_{e\in\Gamma}\int_e[u_h][v] \quad \forall v\in V.
  \end{equation}
Now we resort to the interpolation operator $\Pi: V\rightarrow V_h$, which satisfies the approximation property, for any $v\in V$,
  \begin{equation}
    \| v-\Pi v \|_{L^2(\Omega)}\leq Ch^{\alpha/2}\|v\|_{H^{\alpha/2}(\Omega_h)}.
  \end{equation}
Define $$\|v\|_{E(T)}^2=\|v\|_{H^{\alpha/2}(T)}^2+c_e\|[v]\|_{L^2(\partial T)}^2,$$ where $c_e=1/2$ if $e\in\Gamma_i$ and $c_e=1$ if $e\in\Gamma_b$.
Then $$\|v\|_{E(\Omega_h)}^2=\sum_{T\in\Omega_h}\|v\|_{E(T)}^2,$$ and thus
  \begin{equation}\label{upbd}
     \begin{split}
    |a(e_h,v)|&=|a(e_h,v-\Pi v)|
    \leq \underset{T\in\Omega_h}{\sum}|(R,v-\Pi v)_T|+\underset{e\in\Gamma}{\sum}\int_e\big|[u_h][v-\Pi v]\big|\\
            &\leq C\underset{T\in\Omega_h}{\sum}\|R\|_{L^2(T)}\cdot\|v-\Pi v\|_{L^2(T)}
                +C\underset{e\in\Gamma}{\sum}~\|[u_h]\|_{L^2(e)}\cdot\|v\|_{E(T)}\\
            &\leq C\left(\underset{T\in\Omega_h}{\sum}\eta_T^2\right)^{1/2} \cdot
                   \left(\underset{T\in\Omega_h}{\sum}h_T^{-\alpha}\|v-\Pi v\|_{L^2(T)}^2
                   +\|v\|_{E(T)}^2\right)^{1/2}\\
            &\leq C\left(\underset{T\in\Omega_h}{\sum}\eta_T^2\right)^{1/2} \cdot
                   \left(h^{-\alpha}\|v-\Pi v\|_{L^2(\Omega)}^2+\|v\|_{E(\Omega_h)}^2\right)^{1/2}\\
            &\leq C\left(\underset{T\in\Omega_h}{\sum}\eta_T^2\right)^{1/2} \cdot \|v\|_{E(\Omega_h)}.
    \end{split}
  \end{equation}
  Taking $v=e_h\in H^{\alpha/2}(\Omega_h)$, we have
  \begin{equation}
    \|e_h\|_{E(\Omega_h)}\leq C\left(\underset{T\in{\Omega_h}}{\sum}\eta_T^2\right)^{1/2}.
  \end{equation}
\end{proof}

\begin{lemma}[lower bound]\label{lemma:3.2}
There exists constant $C_2$ depending only on the minimum angle of $\Omega_h$, such that
\begin{equation}
  \eta_T^2\leq C_2(\|u-u_h\|_{E(T)}^2+h_T^\alpha\|R-Q_hR\|_{L^2(T)}^2),
\end{equation}
  where $Q_h$ is the $L^2$ orthogonal projection onto the space of (discontinuous) piecewise polynomials of degree $k-1$.
\end{lemma}

\begin{proof}
  Let $\zeta$ be the bubble function in element $T$, vanishing outside $T$. Then for any polynomial $\phi$ \cite{Brenner1},
  \begin{equation*}
    \|\phi\|_{L^2(T)}^2\leq C\int_T \zeta(\phi)^2dx\leq C\|\phi\|_{L^2(T)}^2.
  \end{equation*}
  Taking $v=\zeta Q_hR \in V_h$ in \eqref{ResEq}, we have
  \begin{equation}\label{bubble}
    \|v\|_{L^2(T)}\leq C\|Q_hR\|_{L^2(T)},
  \end{equation}
  and
  \begin{equation*}
    a(e_h,v)=(R,v).
  \end{equation*}
  Thus,
  \begin{equation*}
    \begin{split}
      \|Q_hR\|_{L^2(T)}^2
      &\leq C\int_T\zeta(Q_hR)^2dx\\
      &\leq C\left(\int_T v(Q_hR-R)dx+\int_T vRdx\right)\\
      &\leq C\left(\int_T v(Q_hR-R)dx+a(e_h,v)\right)\\
      &\leq C\left(\|v\|_{L^2(T)}\cdot\|Q_hR-R\|_{L^2(T)}+\|e_h\|_{H^{\alpha/2}(T)}\cdot h_T^{-\alpha/2}\|v\|_{L^2(T)}\right),
    \end{split}
  \end{equation*}
  where the inverse inequality is used in the last inequality; the proof of the inverse inequality is given in Appendix \ref{invse}.
  Then using \eqref{bubble}
  \begin{equation*}
    h_T^{\alpha/2}\|Q_hR\|_{L^2(T)}\leq C(h_T^{\alpha/2}\|Q_hR-R\|_{L^2(T)}+\|e_h\|_{H^{\alpha/2}(T)}),
  \end{equation*}
  and hence
  \begin{equation*}
  \begin{split}
    h_T^{\alpha/2}\|R\|_{L^2(T)}
    &\leq h_T^{\alpha/2}(\|Q_hR\|_{L^2(T)}+\|Q_hR-R\|_{L^2(T)})\\
    &\leq C(h_T^{\alpha/2}\|Q_hR-R\|_{L^2(T)}+\|e_h\|_{H^{\alpha/2}(T)}).
    \end{split}
  \end{equation*}

  Next, we consider the jump term. Similarly, let $\zeta$ be the bubble function of one edge $e\in T$, vanishing outside $T$.
  Taking $v=\zeta [u_h] \in V_h$ in \eqref{ResEq},
  \begin{equation*}
    \|[u_h]\|_{L^2(e)}^2 \leq C\int_e [u_h]v\\ \leq C((R,v)-a(e_h,v)).
  \end{equation*}
  Then using inverse inequality and trace inequality lead to
  \begin{equation*}
    |(R,v)| \leq C\|R\|_{L^2(T)}\cdot h_T^{1/2}\|v\|_{L^2(e)},
  \end{equation*}
  and
  \begin{equation*}
  \begin{split}
 |a(e_h,v)|
 &\leq C(\|e_h\|_{H^{\alpha/2}(T)}\cdot\|v\|_{H^{\alpha/2}(T)}+\|e_h\|_{L^2(e)}\cdot\|v\|_{L^2(e)})\\
 &\leq C(\|e_h\|_{H^{\alpha/2}(T)}\cdot h_T^{(1-\alpha)/2}\|v\|_{L^2(e)}+\|e_h\|_{L^2(e)}\cdot\|v\|_{L^2(e)}).
 \end{split}
  \end{equation*}
  Therefore, for small $h$,
  \begin{equation*}
  \begin{split}
    \|[u_h]\|_{L^2(e)}
    &\leq C(h_T^{1/2}\|R\|_{L^2(T)}+h_T^{(1-\alpha)/2}\|e_h\|_{H^{\alpha/2}(T)}+\|e_h\|_{L^2(e)})\\
    &\leq C(h_T^{\alpha/2}\|R\|_{L^2(T)}+\|e_h\|_{H^{\alpha/2}(T)}+\|e_h\|_{L^2(e)}),
    \end{split}
  \end{equation*}
  and thus
  \begin{equation*}
    \begin{split}
    \|[u_h]\|_{L^2(e)}^2
    &\leq C(h_T^{\alpha}\|R\|_{L^2(T)}^2+\|e_h\|_{H^{\alpha/2}(T)}^2+\|e_h\|_{L^2(e)}^2)\\
    &\leq C(h_T^{\alpha}\|Q_hR-R\|_{L^2(T)}^2+\|e_h\|_{E(T)}^2).
    \end{split}
  \end{equation*}
The proof is completed by combining all these estimates.
\end{proof}

The goal of adaptive methods is the generation of a mesh which is adapted to the problem such that a given criterion, like a tolerance for the estimated error between exact and discrete solution, is fulfilled by the discrete solution on this mesh. An optimal mesh should be as coarse as possible while meeting the criterion, in order to save computational time and memory requirements. A global refinement of the mesh would lead to the best error reduction, but the amount of new unknowns might be much larger than needed to reduce the error below the given tolerance. We use the Marking Strategy C in \cite{Nochetto2} to deal with the error estimator and the oscillation simultaneously. More strategies can be found in \cite{Siebert1}, like Maximum strategy and Equidistribution strategy.

First, we define 
\begin{equation}
  \mathrm{osc}(T)=h_T^{\alpha/2}\|R-Q_hR\|_{L^2(T)} \quad \forall T\in \Omega_h.
\end{equation}

\textbf{Marking Strategy C:}  Given a parameter $0<\theta_1,\theta_2<1$, construct a subset $\hat{\Omega}_h$ of $\Omega_h$ such that
\begin{equation}
  \underset{T\in\hat{\Omega}_h}{\sum}\eta_T^2\geq \theta_1^2 \underset{T\in\Omega_h}{\sum}\eta_T^2.
\end{equation}
Enlarge $\hat{\Omega}_h$ such that
\begin{equation}
  \underset{T\in\hat{\Omega}_h}{\sum}\mathrm{osc}(T)^2\geq \theta_2^2 \underset{T\in\Omega_h}{\sum}\mathrm{osc}(T)^2.
\end{equation}

Then we will refine the mesh $\Omega_h$ by Marking Strategy C, and show that the error $e_h$ converges to zero in the energy norm at each refinement. Before this, we include the subscripts $j$ to signify the symbol corresponding to the $j$th refinement. We have
\begin{equation}
  \mathrm{osc}_j(T)=h_T^{\alpha/2}\|R_j-Q_jR_j\|_{L^2(T)},
\end{equation}
and
\begin{equation}
  \mathrm{osc}_j(\Omega_h)^2=\underset{T\in\Omega_h}{\sum}\mathrm{osc}_j(T)^2.
\end{equation}

The following two lemmas can be proved similarly as in \cite{Brenner1}. The main difference is the definition of the energy norm. Here we omit the details.
\begin{lemma}[error reduction]\label{lemma:3.3}
  Suppose an element $T\in\Omega_j$ contains a node of $\Omega_{j+1}$ in its interior. Then we have
  \begin{equation}
    \eta_T^2\leq C(\|u_{j+1}-u_j\|_{E(T)}^2+\mathrm{osc}_j(T)^2),
  \end{equation}
  where the positive constant $C$ depends only on the minimum angle of $\Omega_j$.
\end{lemma}

\begin{lemma}[oscillation reduction]\label{lemma:3.4}
 Let $T\in\Omega_j$ be subdivided into elements in $\Omega_{j+1}$ such that
 \begin{equation}
   h_{T'}\leq \gamma h_T \quad \mathrm{if}\ T'\in\Omega_{j+1}\ \mathrm{and}\ T'\subset T,
 \end{equation}
 where the positive constant $\gamma\leq1$. There exist constants $0<\rho_1<1$ and $0<\rho_2$, depending only on $\gamma$, such that
 \begin{equation}
   \mathrm{osc}_{j+1}(\Omega_h)^2\leq\rho_1\mathrm{osc}_j(\Omega_h)^2+\rho_2\|u_{j+1}-u_j\|_{E(\Omega_h)}^2,
 \end{equation}
\end{lemma}

Combining Lemma \ref{lemma:3.3} and Lemma \ref{lemma:3.4} with Galerkin orthogonality
\begin{equation}
  \|u_{j+1}-u_j\|_{E(\Omega_h)}^2=\|u-u_j\|_{E(\Omega_h)}^2-\|u-u_{j+1}\|_{E(\Omega_h)}^2,
\end{equation}
we can easily get the following convergence proposition, which shows that $u_j$ converges to $u$ in the energy norm.
\begin{corollary}
  Let $\{u_j\}_{j\geq1}$ be the sequence of finite element solutions generated by Marking Strategy C. There exist positive constants $\chi$ and $\xi$ such that $0<\xi<1$ and
  \begin{equation}
    \|u-u_{j+1}\|_{E(\Omega_h)}^2+\chi\mathrm{osc}_{j+1}(\Omega_h)^2\leq \xi(\|u-u_j\|_{E(\Omega_h)}^2+\chi\mathrm{osc}_j(\Omega_h)^2).
  \end{equation}
\end{corollary}

\subsubsection{Scheme 2 -- Dual Weighted Residual method (DWR)}\label{section:3.1.3}
Different from the scheme above, the traditional approach to adaptivity aiming at estimating the error with respect to the generic energy norm of the problem, or the global $L^2$ norm, the Dual Weighted Residual method (DWR) \cite{Rolf1} for goal-oriented error estimation aims at economical computation of arbitrary quantities of physical interest. This is typically required in the design cycles of technical applications. `Goal-oriented' adaptivity is designed to achieve these tasks with minimal cost.

When solving the fractional problems, the DWR method is significantly better than the traditional approach since the fractional operator and energy norm is nonlocal. In detail, for left Riemann-Liouville fractional operator, the numerical solution on one element $T$ is affected by all the elements on the left. Therefore, the lower bound is not local absolutely and the over refinement may occur. While DWR multiplies every local error indicator on one element by a weight, consisting of the dual solution. It has the feature of a `generalized' Green function $G(T,T')$, which describes the dependence of the target error quantity $J(e_h)$ concentrated at some element $T$; we select it as $\|e_h\|_{E(T)}$ in the following, on local properties of the data, i.e., the error estimator on element $T'$.

Following the general concept of the DWR method, let $z\in V$ be the solution of the associated dual problem
\begin{equation}\label{equation:3.1.3.0}
  a(\varphi,z)=J(\varphi) \quad \forall \varphi\in V,
\end{equation}
and $z_h\in V_h$ be discontinuous finite element approximation defined by
\begin{equation}\label{equation:3.1.3.1}
  a(\varphi_h,z_h)=J(\varphi_h) \quad \forall \varphi_h\in V_h,
\end{equation}
where the bilinear form $a(\cdot,\cdot)$ is defined as
\[a(u,v)=a_x^\alpha(u,v)+a_y^\beta(u,v)+\sum\limits_{e\in\Gamma}\int_e[u][v].\]
Using this construction together with Galerkin orghogonality, we obtain
\begin{equation}\label{equation:3.1.3.2}
\begin{split}
  J(e_h)&=a(e_h,z)=a(e_h,z-z_h)\\
      &=\underset{T\in\Omega_h}{\sum}(R,z-z_h)+\sum\limits_{e\in\Gamma}\int_e[u_h][z-z_h]\\
      &\leq C\underset{T\in\Omega_h}{\sum}\Big(\|R\|_{L^2(T)}\cdot\|z-z_h\|_{L^2(T)}
        +\|[u_h]\|_{L^2(\partial T)}\|[z-z_h]\|_{L^2(\partial T)}\Big).
      \end{split}
\end{equation}

Thus, we define the local error indicator
\begin{equation}\label{equation:3.1.3.3}
  \eta_T:=\|R\|_{L^2(T)}\cdot\|z-z_h\|_{L^2(T)}+\|[u_h]\|_{L^2(\partial T)}\|[z-z_h]\|_{L^2(\partial T)}.
\end{equation}
Taking $J(\varphi)$ such that $J(e_h)=\|e_h\|_{E(\Omega_h)}$, then we have the global upper bound in energy norm,
\begin{equation}\label{equation:3.1.3.4}
  \|e_h\|_{E(\Omega_h)}=J(e_h)\leq C\underset{T\in\Omega_h}{\sum}\eta_T.
\end{equation}

If we take a rough estimate to $J(e_h)$ by $\|z-z_h\|_{L^2(\Omega)}\leq Ch^{\alpha/2}|z|_{{H^{\alpha/2}}(\Omega_h)}$ and $\|[z-z_h]\|_{L^2(\partial T)}\leq C\|z\|_{L^2(\partial T)}$, then we have, similar to \eqref{upbd}, from (\ref{equation:3.1.3.2}), with
a priori analysis in forms of bounds for $z$,
\[\|e_h\|_{E(\Omega_h)}\leq \Big(\underset{T\in\Omega_h}{\sum}h_T^{\alpha}\|R\|_{L^2(\Omega)}^2+\|[u_h]\|_{L^2(\partial T)}^2\Big)^{1/2}.\]

Here, we get a global posteriori error estimate based on energy norm that is consistent with Scheme 1. In this sense, a posteriori error estimate based on DWR is more meticulous. In order to evaluate the posteriori error representation (\ref{equation:3.1.3.4}), we need information about the discontinuous dual solution $z$. Since in practice, $z$ is not explicitly known, such information has to be obtained either through a priori analysis in form of bounds for $z$ in certain Sobolev norms or through computation by solving the dual problem numerically.

Here, we approximate $z$ by a high-order method. We take $V_h$ as the linear discontinuous finite element and solve the dual problem by using quadratic discontinuous finite element on the current mesh yielding an approximation $z_h^{(2)}\in V_h^{(2)}$ to $z$, and $z_h$ can be got by linear interpolation of $z$. This yields the approximate local error indicator
\begin{equation}
  \eta_T\approx \|R\|_{L^2(T)}\cdot\|z_h^{(2)}-\Pi_hz_h^{(2)}\|_{L^2(T)}.\nonumber
\end{equation}

\begin{remark}
  In addition to the above points that a posteriori error estimate based on DWR is better than a global posteriori error estimate based on energy norm, the former has advantages when the derivative on $x-$direction and on $y-$direction is different, i.e., $\alpha\neq\beta$.
  When discussing a posteriori error estimate based on energy norm, we take $\alpha=\beta$ in Eq (\ref{equation:3.0}) for convenience. Actually, if $\alpha<\beta$, the indicator is not easy to select. We must take the exponent of $h_T$ in the indicator $\eta_T:=h_T^{\alpha/2}\|R\|_{L^2(T)}$ to be $\alpha/2$ to guarantee the upper bound and take it to be $\beta/2$ to guarantee the lower bound. A posteriori error estimate based on DWR avoids this problem, and it is still effective for complex problems.
\end{remark}

\subsubsection{Numerical experiment}\label{section:3.1.4}
In this section we will present some numerical experiments using the two schemes given above. We compare them with uniformly refinement and with each other.

\begin{example}\label{exam:3.1}
Consider the 1D fractional equation on the domain $\Omega=[0,2]$ with $\alpha=0.8$,
\begin{equation}
\left\{ \begin{array}{ll}
  {}_0D_x^{\alpha}u=f  \quad\Omega\\\nonumber
  u=0  \quad\quad\quad \mathbb{R}\backslash\Omega,
\end{array} \right.
\end{equation}
and its dual problem is
\begin{equation}
\left\{ \begin{array}{ll}
  {}_xD_2^{\alpha}u=f  \quad\Omega\\\nonumber
  u=0  \quad\quad\quad \mathbb{R}\backslash\Omega.
\end{array} \right.
\end{equation}
\end{example}

The source term $f$ is chosen such that the exact solution is $u=\left(1-(x-1)^2\right)^\gamma$, $\gamma=0.7$, which has poor regularity near the boundary, and we use the discontinuous piecewise linear function for approximation. 

\begin{figure}[!htb]
  \centering
  \includegraphics[scale=0.4]{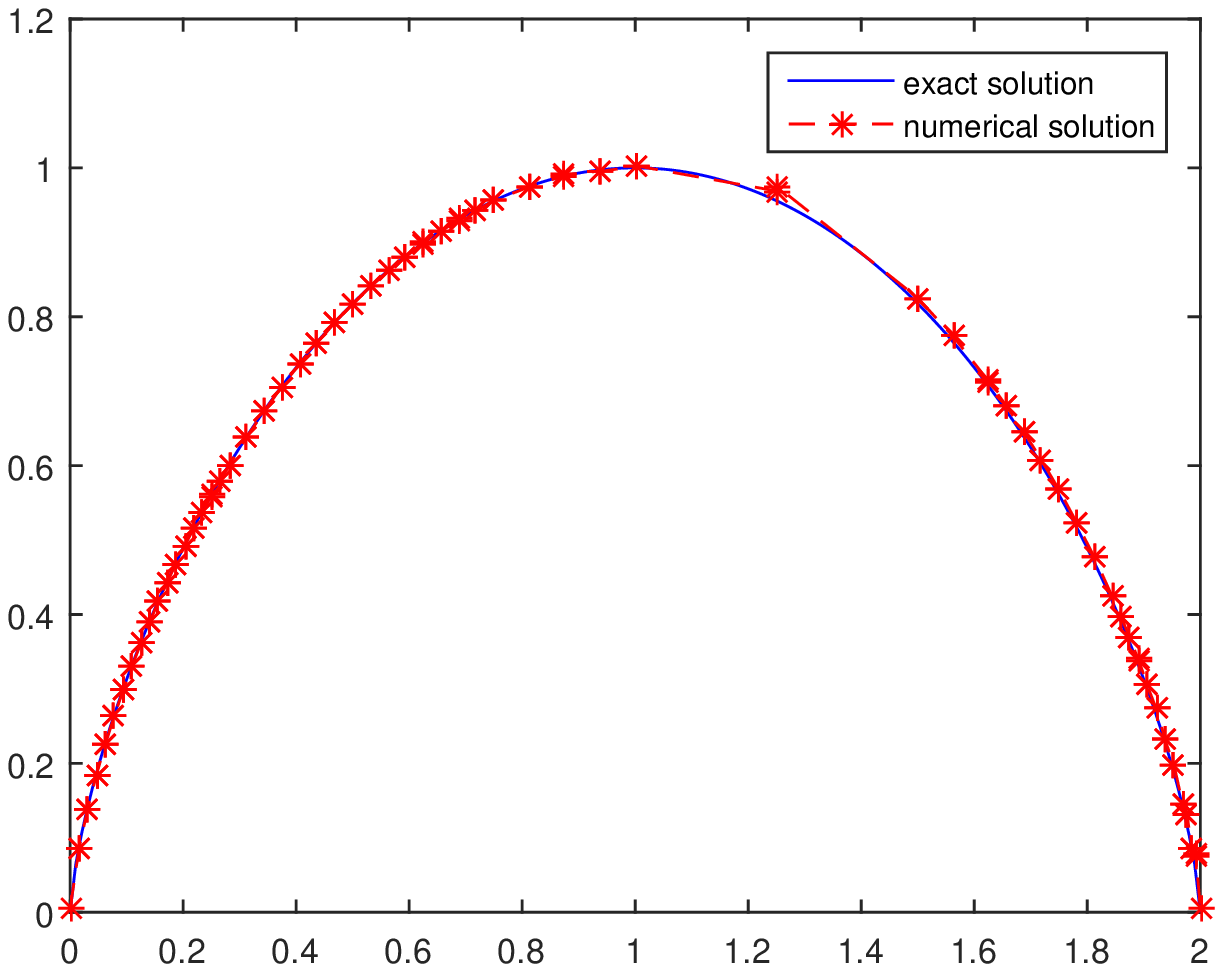}
  \includegraphics[scale=0.4]{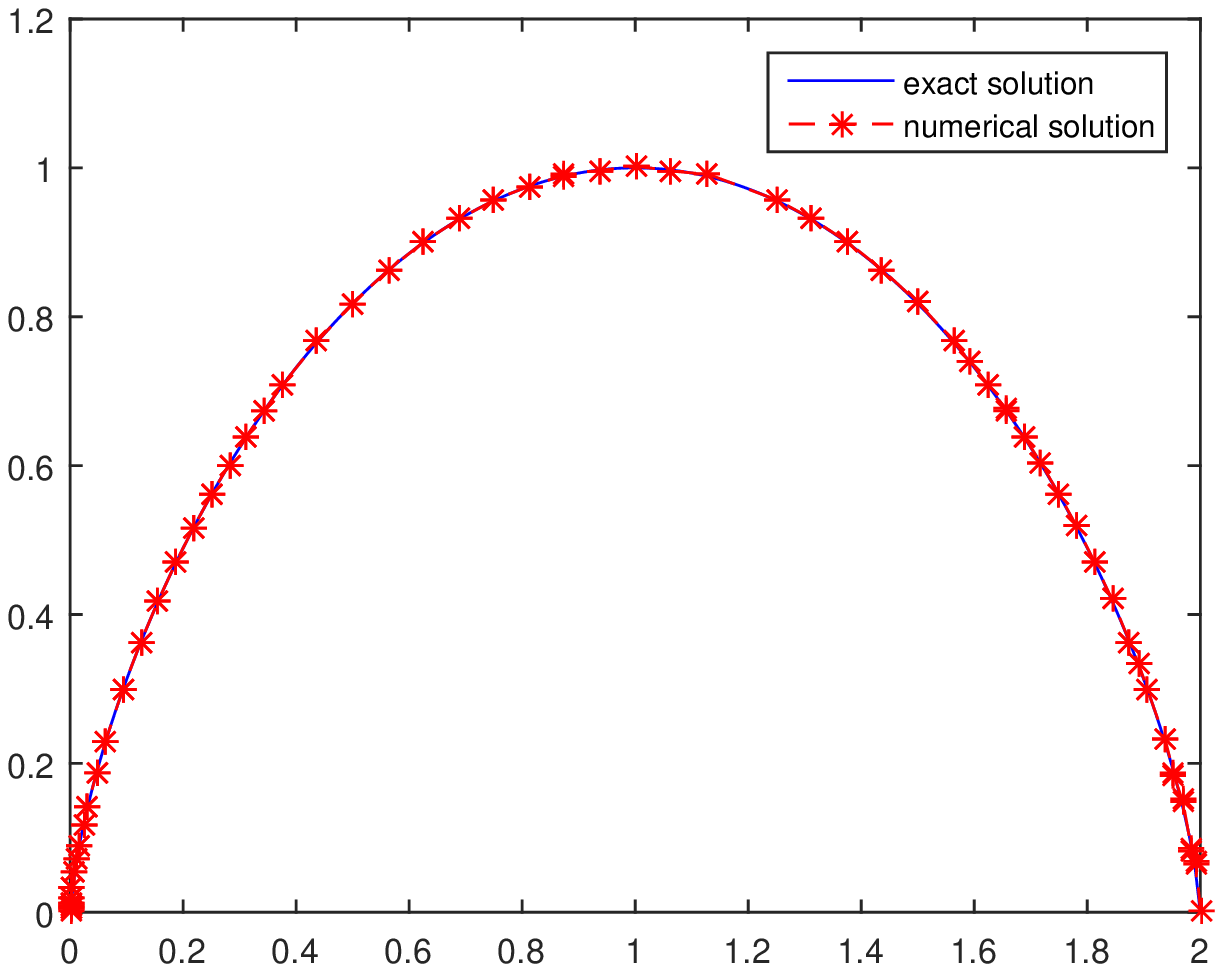}\\
  \caption{Adaptive refinement of Example \ref{exam:3.1}: energy-norm indicator on 60 elements (left) and weighted indicator on 58 elements (right).}\label{fig:3.1.1}
\end{figure}
\begin{figure}[!htb]
  \centering
  \includegraphics[scale=0.6]{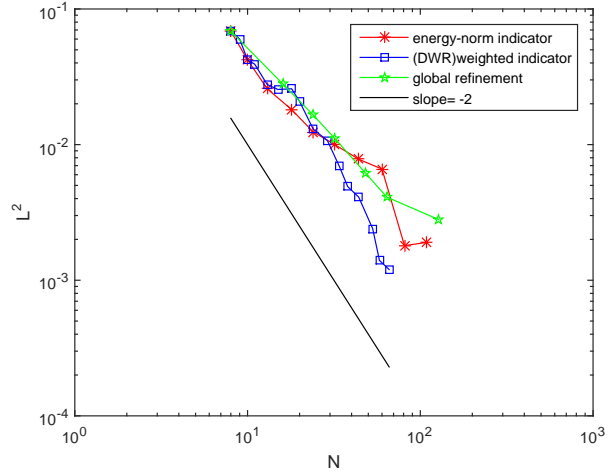}\\
  \caption{$L^2$ error versus number of elements $N$ of Example \ref{exam:3.1}, for uniform refinement, the energy norm indicator, the weighted indicator obtained by the DWR approach.}\label{fig:3.1.2}
\end{figure}

This is a boundary layer problem, i.e., the solution has less regularity around the endpoints of domain $[0,2]$, where the mesh should be finer. We initially divide the interval $[0,2]$ into 8 cells uniformly, then refine the mesh based on the energy-norm indicator in Figure \ref{fig:3.1.1} (left) and the weighted-norm indicator in Figure \ref{fig:3.1.1} (right). The numerical solution approximates the exact solution both very well obviously. And from an intuitive point of view, the right one is better than the left one. Besides, we compare the two kinds of refinements with uniform refinement, and show their convergence rates. A reference line is provided which shows the optimal convergence rate $N^{-2}$. As shown in Figure \ref{fig:3.1.2}, the uniform refinement is the worst while the refinement based on the weighted indicator is the best, almost achieving the optimal convergence rate.

\begin{example}\label{exam:3.2}
Consider the 2D fractional equation on the domain by $\Omega:=[0,2]\times[0,2]$ with $\alpha=0.2$, $\beta=0.8$,
\begin{equation}
\left\{ \begin{array}{ll}
{}_0D_x^{\alpha}u+{}_xD_2^{\alpha}u+{}_0D_y^{\beta}u+{}_yD_2^{\beta}u=f   \quad\Omega  \\
                                            \qquad\qquad\qquad\qquad\qquad\qquad\quad    u=0    \quad\mathbb{R}^2\backslash\Omega.
\end{array} \right.
\end{equation}
and its dual problem is itself.
The source term $f$ is chosen such that the exact solution writes
\[u(x,y)=x(x-2)y(y-2)\arctan\Big(\frac{\sqrt{x^2+y^2}-2}{0.05}\Big).\]
\end{example}

\begin{figure}[!htb]
  \centering
  \includegraphics[scale=0.5]{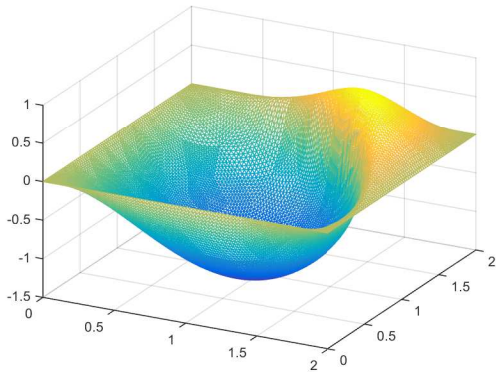}\\
  \caption{Surface plot of the exact solution of Experiment \ref{exam:3.2}, and the steep region around the arc $x^2+y^2=4$.}\label{fig:3.1.0}
\end{figure}
\begin{table}[!htb]
\centering
\caption{Example \ref{exam:3.2}: Uniform refinement.}
  \begin{tabular}{*{3}{|c}|}\hline
    K  & $\|u-u_k\|_{L^2}$ & $\|u-u_k\|_E$  \\\hline
   32  & 0.3164 & 1.1565 \\
   72  & 0.1908 & 0.9010 \\
   128 & 0.1257 & 0.6794 \\
   200 & 0.0935 & 0.5264 \\
   288 & 0.0755 & 0.4357 \\
   392 & 0.0637 & 0.3803 \\
   512 & 0.0528 & 0.3362 \\\hline
\end{tabular}\vspace{0.2cm}
\label{table:3}
\end{table}
\begin{table}[!htb]
  \centering
  \caption{Example \ref{exam:3.2}: Refinement based on the energy-norm indicator.}
  \begin{tabular}{*{6}{|c}|}\hline
  k  & K  & $\|u-u_k\|_{L^2}$  &$\|u-u_k\|_E$&  $\eta$  & $I_{eff}$ \\\hline
  1  & 8  & 0.6752             & 1.2608      &  3.5025  & 2.7781 \\
  2  & 15 & 0.3930             & 1.0819      &  2.0108  & 1.8586 \\
  3  & 28 & 0.2614             & 0.7705      &  1.6311  & 2.1168 \\
  4  & 42 & 0.1652             & 0.6573      &  1.1697  & 1.7796 \\
  5  & 86 & 0.1169             & 0.4383      &  0.9193  & 2.0973 \\
  6  & 152 & 0.0872            & 0.4074      &  0.6599  & 1.6198 \\
  7  & 201 & 0.0684            & 0.3023      &  0.5550  & 1.8361 \\\hline
  \end{tabular}\vspace{0.2cm}
  \label{table:4}
\end{table}
\begin{table}[!htb]
  \centering
  \caption{Example \ref{exam:3.2}: Refinement based on the weighted indicator.}
  \begin{tabular}{*{6}{|c}|}\hline
  k  & K  & $\|u-u_k\|_{L^2}$  &$\|u-u_k\|_E$&  $\eta$  & $I_{eff}$ \\\hline
  1  & 8  & 0.6752             & 1.2608      &  1.0679  & 0.8470 \\
  2  & 15 & 0.3930             & 1.0819      &  0.7608  & 0.7032 \\
  3  & 28 & 0.2614             & 0.7705      &  0.3620  & 0.4698 \\
  4  & 36 & 0.1650             & 0.6408      &  0.2558  & 0.3992 \\
  5  & 63 & 0.1242             & 0.4852      &  0.2260  & 0.4658 \\
  6  & 110 & 0.0933            & 0.4174      &  0.1603  & 0.3840 \\
  7  & 166 & 0.0551            & 0.2766      &  0.1146  & 0.4143 \\\hline
  \end{tabular}\vspace{0.2cm}
  \label{table:5}
\end{table}

Being the same as the case of 1D, we compare the three kinds of refinements: uniform refinement, the refinement based on energy-norm indicator, and weighted indicator. This exact solution has less regularity inside the domain, i.e., interior layer. Figure \ref{fig:3.1.0} is the shape of the exact solution. The adaptive mesh will be finer in the steep region. The experiment datum are  provided in Table \ref{table:3}, Table \ref{table:4}, and Table \ref{table:5}. We denote the total error indicator by $\eta$, and the effectiveness index $I_{eff}=\eta/\|e\|_{E(\Omega_h)}$. If the effectiveness index $I_{eff}$ remains roughly constant in different meshes, the indicator $\eta$ approximates the true error $e_h$ well. By comparing these tables, adaptive refinement is significantly better than uniform refinement. The weighted indicator is slightly better than the energy-norm indicator since the $L^2$ norm error and energy-norm error of the former is smaller than the latter, and the effectiveness index $I_{eff}$ of the former changes more moderately than the latter. In Figure \ref{fig:3.1.3}, two kinds of adaptive meshes are presented, and both refine the steep region. The right one (based on the weighted indicator) is slightly better than the left one (based on the energy norm indicator).
\begin{figure}[!htb]
  \centering
  \includegraphics[scale=0.4]{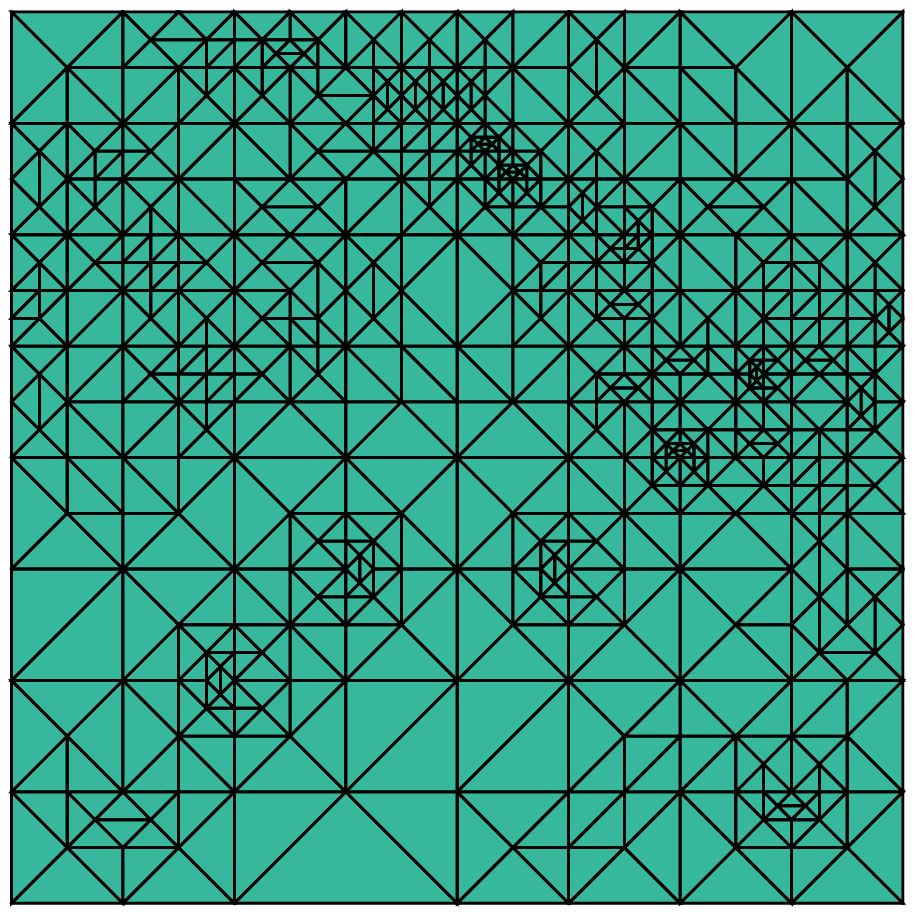}
  \includegraphics[scale=0.4]{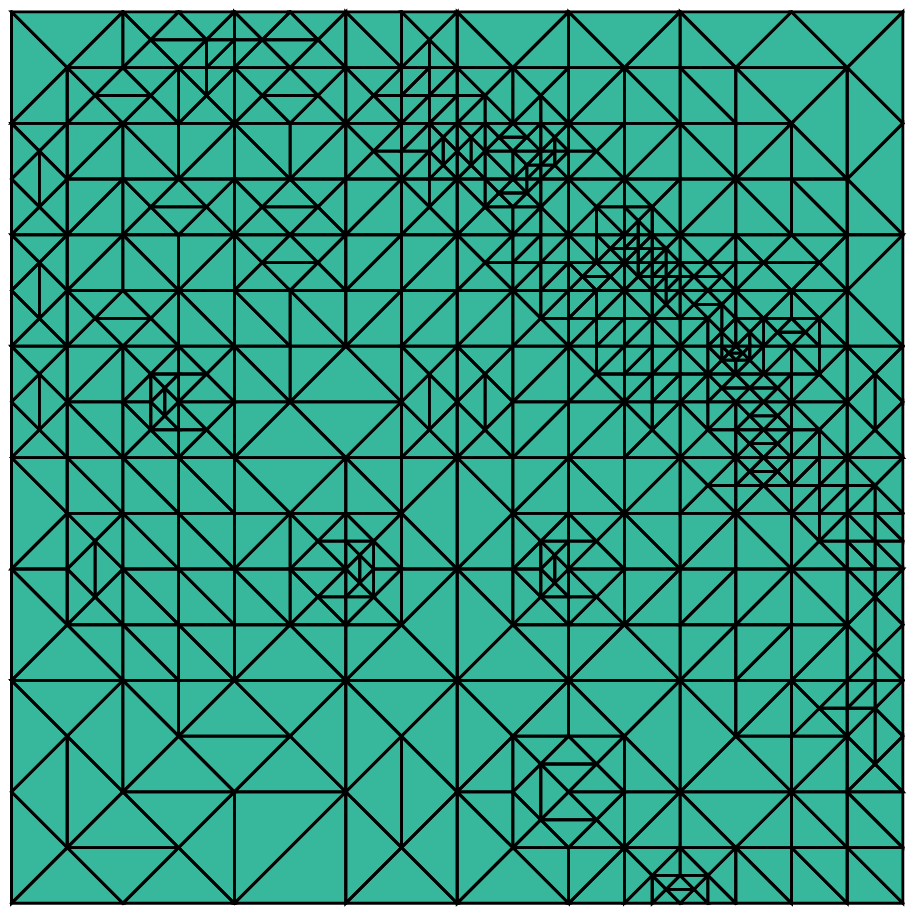}\\
  \caption{The mesh after 11 iterations (973 cells) based on energy-norm indicator on the left and the mesh after 11 iterations (903 cells) based on weighted indicator on the right.}\label{fig:3.1.3}
\end{figure}

\subsection{Evolution Equation}\label{section:3.2}
In this section, we consider the time dependent tempered fractional equation:
\begin{equation}\label{equation:3.2.1}
\left\{ \begin{array}{ll}
\partial_t u+\textbf{b}\cdot\nabla u -\kappa_1\nabla_x^{\alpha,\lambda}u-\kappa_2\nabla_y^{\beta,\lambda}u=f
& (\textbf{x},t)\in\Omega\times J \\
u(\textbf{x},0)=u_0(\textbf{x}) &  \textbf{x}\in\Omega \\
u(\textbf{x},t)=0   &  (\textbf{x},t)\in\mathbb{R}^2\backslash\Omega\times J
\end{array} \right.
\end{equation}
Different from the stationary equation, the mesh here is adapted to the solution in every time step using a posteriori error indicator, i.e., the adaptive algorithm solving the evolution equation at the $n-$th time step reads as
\[ \mathrm{Solve}\rightarrow\mathrm{Estimate}\rightarrow\mathrm{Refine}/\mathrm{Coarsen}.
\]
Here the refinement/coarsening procedure includes both the mesh and time-step size modifications. In this paper, we propose the following algorithm, similar to \cite{Ming1}, to modify the time-step size $\tau_n$ and mesh $\Omega_h^n$ starting from the initial time-step size $\tau_{n,0}=\tau_{n-1}$ and initial mesh $\Omega_h^{n,0}=\Omega_h^{n-1}$:

1. Refine the time-step size $\tau_{n,0}$ to the final time-step size $\tau_n$ such that the associated time error indicators are less than the prescribed tolerances.

2. Refine/Coarsen the mesh $\Omega_h^{n,0}$ to the final mesh $\Omega_h^n$ such that the associated space error indicators are less than the prescribed tolerances.

3. Enlarge the initial time-step size $\tau_{n+1,0}$ for next time step if the current time error indicator is much less than the tolerances.

\subsubsection{A posteriori error analyses}\label{section:3.2.1}
While the previous section is concentrated on the spatial mesh refinement based on the local error indicator, this section describes the process of evolution of the solution and the time-step size based on the time error indicator. Then we introduce the time and space local error indicator, and prove global upper bound and local lower bound. For convenience, we take $\mathbf{b}=\mathbf{0}$, $\alpha=\beta$,
 and let $V_h^n$ indicate the usual space of linear discontinuous finite element in $\Omega^n$ and piecewise constant discontinuous in $\tau_n$ \cite{Rolf1}. Let $u\in V$ be the solution of
\begin{equation}\label{equation:3.2.1.1}
  \left(\frac{\partial}{\partial t}u,v\right)+a(u,v)-\kappa(u,v)=(f,v) \quad \forall v\in V,
\end{equation}
and $u_h^n$ be the fully discrete discontinuous finite element approximation defined by
\begin{equation}\label{equation:3.2.1.2}
  (\partial_t u_h^n,v_h)+a(u_h^n,v_h)-\kappa(u_h^n,v_h)+(u_h^n-u_h^{n-1},v_h)=(\bar{f}^n,v_h) \quad \forall v_h\in V_h,
\end{equation}
where $\partial_tu_h^n=\frac{u_h^n-u_h^{n-1}}{\tau_n}$ and $\bar{f}^n=\frac{1}{\tau_n}\int_{t^{n-1}}^{t^n}f(\mathbf{x},t)dt$.

\begin{theorem}[upper bound]
  For any integer $1\leq m\leq N$, there exists a positive constant $C$ depending only on the minimum angle of meshes $\Omega_h^n,\ n=1,2,\cdots,m,$ such that the posteriori error estimate
  \begin{equation*}
  \begin{split}
\|u^m-u_h^m\|_{L^2(\Omega)}^2&+\sum\limits_{n=1}^m\int_{t^{n-1}}^{t^n}\|u^m-u_h^m\|_{E(\Omega_h^n)}^2dt\\
      &\leq\|u_0-u_h^0\|_{L^2(\Omega)}^2+\sum\limits_{n=1}^m\tau_n(\eta_\mathrm{time1}^n+\eta_\mathrm{time2}^n)+C\sum\limits_{n=1}^m\tau_n\eta_\mathrm{space}^n,
  \end{split}
  \end{equation*}
  holds, where the time error indicator and space error indicator are given by
  \[\eta_\mathrm{time1}^n=\frac{1}{\tau_n}\int_{t^{n-1}}^{t^n}\|f-\bar{f}^n\|_{L^2(\Omega)}^2dt,  \qquad
  \eta_\mathrm{time2}^n=\|u_h^n-u_h^{n-1}\|_{L^2(\Omega)}^2,
  \]
  \[\eta_\mathrm{space}^n=\underset{T\in\Omega_h^n}{\sum}\eta_T^n
  =\underset{T\in\Omega_h^n}{\sum} \Big(h_T^\alpha\|R^n\|_{L^2(T)}^2+\|[u_h]\|_{L^2(\partial T)}^2 \Big).\]
\end{theorem}

\begin{proof}
From (\ref{equation:3.2.1.1}) and (\ref{equation:3.2.1.2}), for a.e. $t\in(t^{n-1},t^n]$, and for any $v\in V$, we have
\begin{equation*}
\begin{split}
  (\partial_t(u-&u_h^n),v)+a(u-u_h^n,v)-\kappa(u-u_h^n,v)\\
  &=(f-\bar{f}^n,v)+(R^n,v)-\sum\limits_{e\in\Gamma}[u_h][v]\\
  &=(f-\bar{f}^n,v)+(R^n,v-\Pi^n v)-\sum\limits_{e\in\Gamma}\int_e[u_h][v-\Pi v]+(u_h^n-u_h^{n-1},\Pi^n v).
  \end{split}
\end{equation*}
where $R^n=\bar{f}^n-\frac{u_h^n-u_h^{n-1}}{\tau_n}-\kappa_1\kappa_\alpha D_x^{\alpha,\lambda}u_h^n
-\kappa_2\kappa_\alpha D_y^{\alpha,\lambda}+\kappa u_h^n$ and $\Pi^n$ is the interpolation operator.

Taking $v=u-u_h^n:=e^n$, similar to \eqref{upbd}, we have
\begin{equation*}\label{equation:3.2.1.3}
\begin{split}
   \frac{1}{2}\frac{d}{dt}\|e^n\|_{L^2(\Omega)}^2+\|e^n\|_{E(\Omega_h)}^2
    \leq&~\kappa\|e^n\|_{L^2(\Omega)}^2+(f-\bar{f}^n,e^n)+(R^n,e^n-\Pi^n e^n)\\
        &+\sum\limits_{e\in\Gamma}\int_e|[u_h][e^n-\Pi e^n]|+(u_h^n-u_h^{n-1},\Pi^n e^n)\\
    \leq&~ C\|e^n\|_{L^2(\Omega)}^2+\frac{1}{2}\|f-\bar{f}^n\|_{L^2(\Omega)}^2+\frac{1}{2}\|u_h^n-u_h^{n-1}\|_{L^2(\Omega)}^2\\
    &+C\underset{T\in\Omega_h^n}{\sum} \Big(h_T^\alpha\|R^n\|_{L^2(\Omega)}^2+\|[u_h]\|_{L^2(\partial T)}^2\Big)  +\frac{1}{2}\|e^n\|_{E(\Omega_h)}^2.
    \end{split}
\end{equation*}
Integrating the above formula in time from $t^{n-1}$ to $t^n$ and summing over $n$ from $1$ to $m$, with the discrete Gr\"{o}nwall inequality, we complete the proof.
\end{proof}

Next, we prove the lower bound to ensure over-refinement will not occur based on our space error indicator. Let $u_\ast^n\in V$ be the solution of auxiliary problem \cite{Ming1}.
\begin{equation}\label{equation:3.2.1.4}
  \left(\frac{u_\ast^n-u_h^{n-1}}{\tau_n},v\right)+a(u_\ast^n,v)-\kappa(u_\ast^n,v)=(\bar{f}^n,v)\qquad \forall v\in V.
\end{equation}
Note that for fixed time-step size $\tau_n$, by adapting the mesh $\Omega_h^n$, we are essentially controlling the error between $u_h^n$ and $u_\ast^n$, not between $u_h^n$ and the exact solution $u$. Based on this observation, we have the following analyses.
\begin{theorem}[lower bound]
  There exist constants $C_2, C_3>0$ depending only on the minimum angle of $\Omega_h^n$ such that for any $T\in\Omega_h^n$, the following estimate holds:
  \begin{equation}\label{equation:3.2.1.5}
    \eta_T^n \leq C_2\mathrm{osc}^n(T)^2+C_3\hat{C}_n\left(\frac{1}{\tau_n}\|u_\ast^n-u_h^n\|_{L^2(T)}^2+\|u_\ast^n-u_h^n\|_{E(T)}^2\right),
  \end{equation}
  where $\hat{C}_n=\underset{T\in\Omega_h^n}{\max}(h_T^\alpha/\tau_n)$ and $\mathrm{osc}^n(T)^2=h_T^{\alpha}\|Q^nR^n-R^n\|_{L^2(T)}^2$.
\end{theorem}
\begin{proof}
  Let $p$ be a node of $\Omega_h^n$ interior to $T$ and $\xi$ be a piecewise linear function with respect to $\Omega_h^n$ that equals $1$ at $p$ and $0$ at all the other nodes. Then $v=\xi Q^nR^n$ belongs to $V_h^n$ and vanishes outside $T$. Thus
  \begin{equation*}\label{equation:3.2.1.6}
    \begin{split}
      \|Q^nR^n\|_{L^2(T)}^2&\leq C\int_T\xi(Q^nR^n)^2dx\\
                           &=C\left(\int_T(Q^nR^n-R^n)v+\int_TR^nvdx\right).
    \end{split}
  \end{equation*}
  Combining with (\ref{equation:3.2.1.4}), we find
  \begin{equation*}\label{equation:3.2.1.7}
  \begin{split}
    \int_TR^nvdx=&\left(\frac{u_\ast^n-u_h^n}{\tau_n},v\right)
             +\kappa_1\kappa_\alpha\big(D_x^{\alpha,\lambda}(u_\ast^n-u_h^n),v\big)\\
             &+\kappa_2\kappa_\alpha\big(D_y^{\alpha,\lambda}(u_\ast^n-u_h^n),v\big)
             -\kappa\big((u_\ast^n-u_h^n),v\big)\\
             \leq&~ C\left(\frac{1}{\tau_n}\|u_\ast^n-u_h^n\|_{L^2(T)}
             +h_T^{-\alpha/2}\|u_\ast^n-u_h^n\|_{E(T)}
             +\|u_\ast^n-u_h^n\|_{L^2(T)}\right)\\&~\cdot\|Q^nR^n\|_{L^2(T)}\\
             \leq&~ C\hat{C}_n^{1/2}h_T^{-\alpha/2}\left(\frac{1}{\tau_n}\|u_\ast^n-u_h^n\|_{L^2(T)}^2+\|u_\ast^n-u_h^n\|_{E(T)}^2\right)^{1/2}
             \cdot\|Q^nR^n\|_{L^2(T)}.
  \end{split}
  \end{equation*}
  Therefore, we have
  \begin{equation*}\label{equation:3.2.1.8}
  \begin{split}
    h_T^{\alpha/2}\|Q^nR^n\|_{L^2(T)}\leq &~Ch_T^{\alpha/2}\|Q^nR^n-R^n\|_{L^2(T)}\\
    &+C\hat{C}_n^{1/2}\left(\frac{1}{\tau_n}\|u_\ast^n-u_h^n\|_{L^2(T)}^2+\|u_\ast^n-u_h^n\|_{E(T)}^2\right)^{1/2}.
  \end{split}
  \end{equation*}
  Thus,
  \begin{equation*}\label{equation:3.2.1.9}
  \begin{split}
    h_T^{\alpha}\|R^n\|_{L^2(T)}^2\leq &~Ch_T^{\alpha}\|Q^nR^n-R^n\|_{L^2(T)}^2\\
    &+C\hat{C}_n\left(\frac{1}{\tau_n}\|u_\ast^n-u_h^n\|_{L^2(T)}^2+\|u_\ast^n-u_h^n\|_{E(T)}^2\right).
    \end{split}
  \end{equation*}
  The term $\|[u_h]\|_{L^2(\partial T)}^2$ can be bounded by the right hand side similarly.

  Combining the above estimates implies (\ref{equation:3.2.1.5}).
\end{proof}

\subsubsection{Numerical experiment}\label{section:3.2.2}
In the last part, we mainly observe the performance of spatial adaptive mesh, and give the specific error of different cases on experiments. Here we present some numerical experiments using the algorithm of evolution equations given above. We focus on observing the evolution of exact solution by taking it with different regularity at different time, which will indicate the effectiveness of the indicator.

\begin{example}\label{exam:4}
  Consider the equation (\ref{equation:3.2.1}) on the domain $\Omega=[0,2]\times[0,2]$, with $\mathbf{b}=\mathbf{0}, \kappa_1=0.1, \kappa_2=0.2, \lambda=0.2$, and $\alpha=\beta=0.8$. The source term $f$ is chosen such that the exact solution writes
  \[u(x,y)=x(x-2)y(y-2) e^{-((x-t)^2+(y-t)^2)/0.005}.\]
\end{example}
For this example, the solution has less regularity near the point $(t,t)$, where we expect to observe more refinement. The adaptation process yields the meshes shown in Figure \ref{fig:4.1}, which is in consistent with the theoretical result completely.
\begin{figure}[!htb]
  \centering
  \includegraphics[scale=0.4]{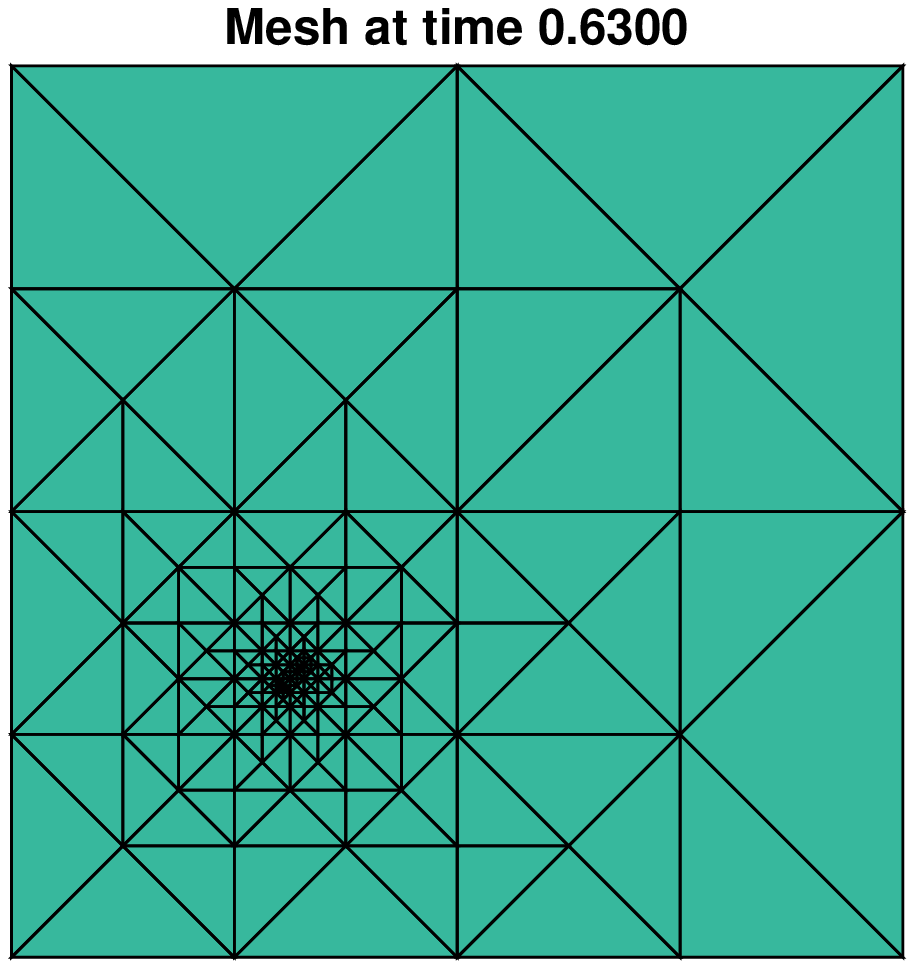}
  \includegraphics[scale=0.4]{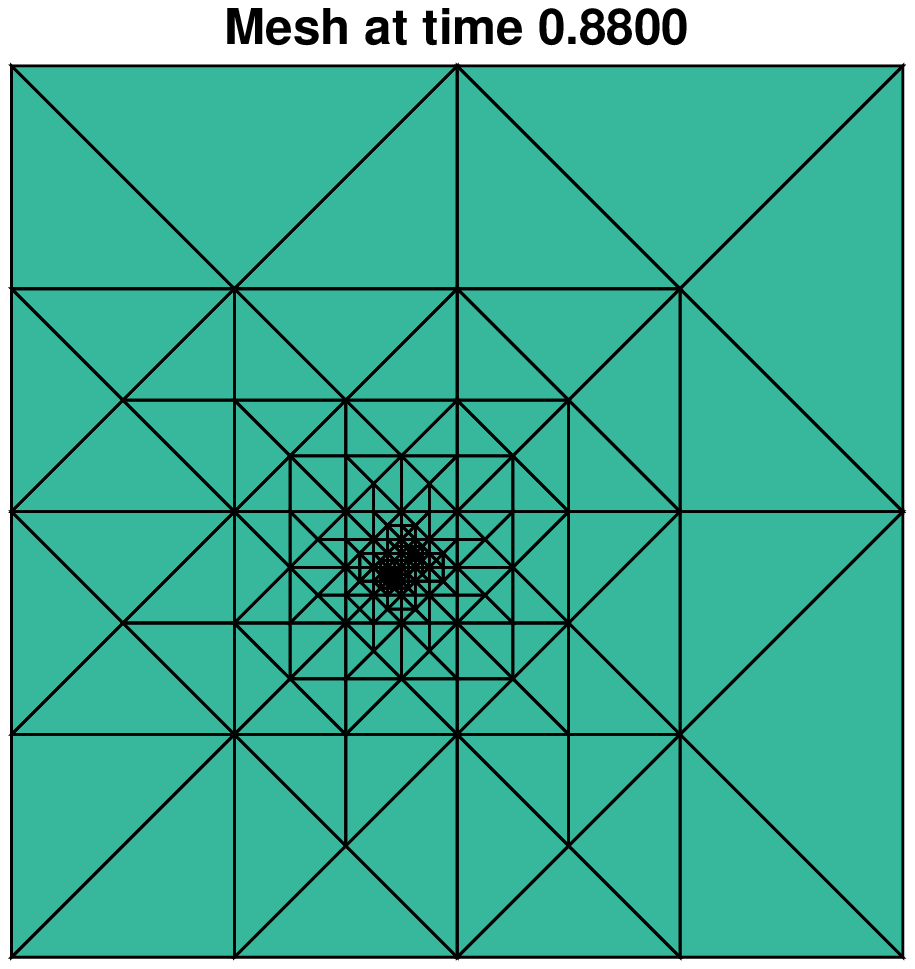} \\
  \includegraphics[scale=0.4]{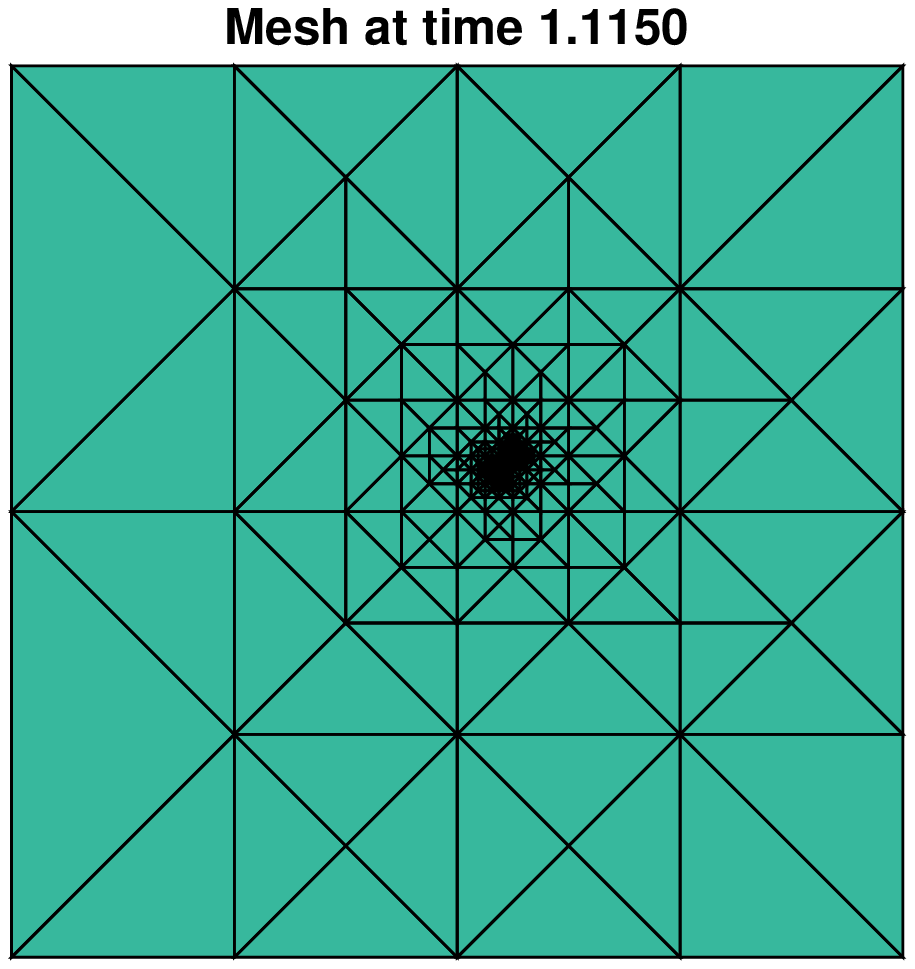}
   \includegraphics[scale=0.4]{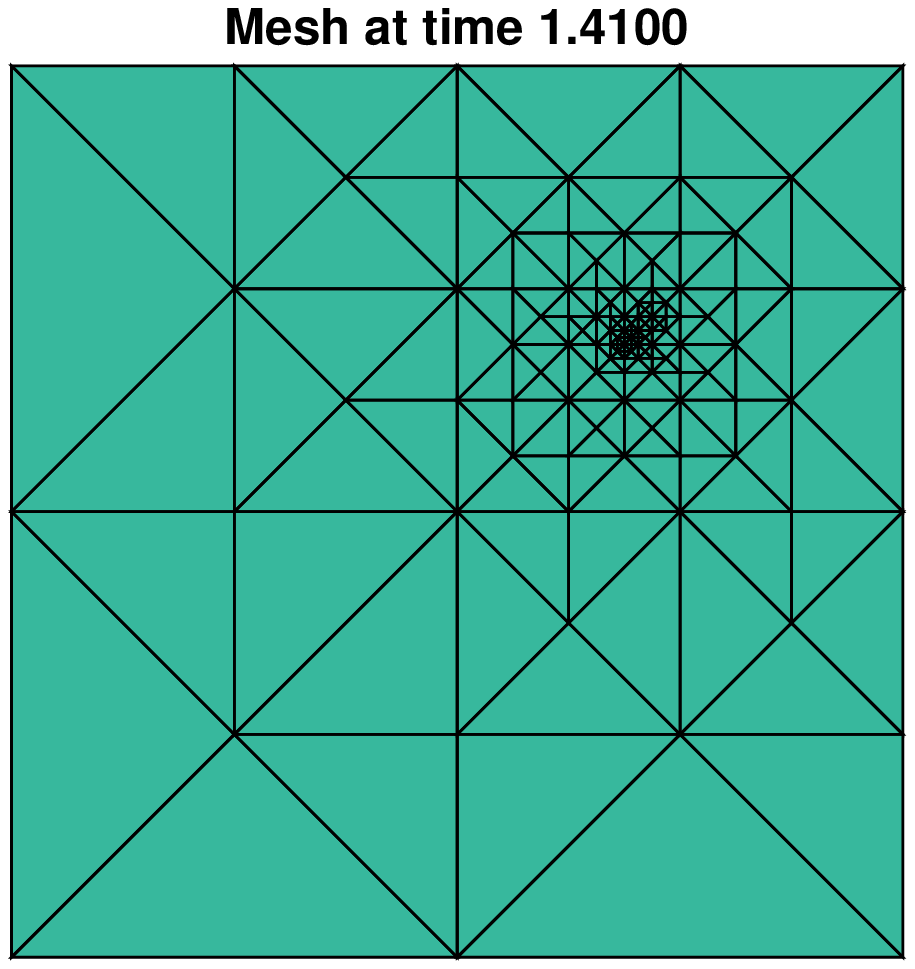}\\
  \caption{Mesh refinement at different time.}\label{fig:4.1}
\end{figure}

\section{Conclusion}
We discuss the DG methods for the two dimensional time dependent space tempered fractional convection-diffusion equations, especially their adaptivity; and the provided strategy for designing adaptive schemes works for the general PDEs with fractional operators.
DG methods are superior to finite element method in many ways, but the stability of the discrete schemes needs more attention. The interior boundaries are usually connected by two ways: interior penalty and numerical flux.
 We use the former to deal with the diffusion term and the latter to the convection term. The stability and convergence analyses are explicitly provided. The theoretical results are confirmed by numerical experiments. For the adaptivity of the DG methods, we consider two schemes of the stationary problem, i.e., a posteriori error estimator based on the traditional energy norm and another one based on the dual weighted residual. The numerical experiments confirm the advantage of the DWR method. Finally, we consider the fractional evolution problem. A posteriori error estimate is provided and the indicator is designed; its effectiveness is displayed by numerical simulations.
 

\section{Acknowledgements}
This work was supported by the National Natural Science Foundation of China under Grant No. 11671182.

\begin{appendix}
\section{The inverse inequation}\label{invse}

\begin{lemma}
  Let $\alpha>0$, $v\in V_k$ and $T\in\Omega_h$. Then there exists a constant $C$ depending only on $\alpha$ and the shape regularity such that $$\|v\|_{H^\alpha(T)}\leq Ch_T^{-\alpha}\|v\|_{L^2(T)},$$ where $h_T$ is the diameter of $T$.
\end{lemma}
\begin{proof}
  First, we consider the one dimensional case, i.e., $T=[0,h]$. Let $\hat{T}=[0,1]$ be the reference triangle. If $v$ is a function defined on $T$, then $\hat{v}$ is defined on $\hat{T}$ by $\hat{v}(\hat{x})=v(h\hat{x})~~\forall\hat{x}\in\hat{T}$.

  For $v$ vanishing outside of $T$, there exists
  \begin{equation}\label{inv1d}
    |\hat{v}|_{H^\alpha(\hat{T})}=h^{\alpha-1/2}|v|_{H^\alpha(T)}.
  \end{equation}
  Actually, taking $x=hy,~\xi=hs$
  \begin{equation*}
    \begin{split}
      {}_0I_x^\alpha v(x)
      &=\frac{1}{\Gamma(\alpha)} \int_0^x (x-\xi)^{\alpha-1}v(\xi)d\xi\\
      &=\frac{h^\alpha}{\Gamma(\alpha)} \int_0^y (y-s)^{\alpha-1}v(hs)ds\\
      &=h^\alpha{}_0I_y^\alpha v(hy).
    \end{split}
  \end{equation*}
  Similarly,
  $$ {}_0D_x^\alpha v(x) =h^{-\alpha}{}_0D_y^\alpha v(hy).$$
  Thus, \eqref{inv1d} is implied by
  \begin{equation*}
    \begin{split}
      |v|_{H^\alpha(T)}^2
      &=\int_0^h |{}_0D_x^\alpha v(x)|^2dx\\
      &=\int_0^1 |h^{-\alpha} {_0D_y^\alpha} v(hy)|^2d(hy)\\
      &=h^{1-2\alpha} |\hat{v}|_{H^\alpha(\hat{T})}^2.
    \end{split}
  \end{equation*}
  On the other hand,
  \begin{equation*}
    \|\hat{v}\|_{L^2(\hat{T})}=h^{-1/2}\|v\|_{L^2(T)}.
  \end{equation*}
  Using the equivalence of any two norms on the finite dimensional function space $V_k$,
  \begin{equation}
    |v|_{H^\alpha(T)}=h^{1/2-\alpha}|\hat{v}|_{H^\alpha(\hat{T})}\leq Ch^{1/2-\alpha}\|\hat{v}\|_{L^2(\hat{T})}=Ch^{-\alpha}\|v\|_{L^2(T)}.
  \end{equation}

  Next, we consider the case that $T$ is a triangle.
  Let $\hat{T}$ be a reference triangle, and $F$ be an affine map from $\hat{T}$ onto $T$. If $v$ is a function defined on $T$, then $\hat{v}$ is defined on $\hat{T}$ by $\hat{v}(\hat{x})=v(x)~~\forall F(\hat{x})=x$.

  Note that $$|v|_{H^\alpha(T)}^2=\|_aD_x^\alpha u\|_{L^2(T)}^2+\|_cD_y^\alpha u\|_{L^2(T)}^2,$$ in which two terms can be decomposed $\alpha-$derivative in $x/y$ direction and no-derivative in $y/x$ direction.

  Then inspired by the one dimensional case above, one obtains
  \begin{equation}
    |v|_{H^\alpha(T)}\leq Ch^{1/2-\alpha+1/2}|\hat{v}|_{H^\alpha(\hat{T})}
    \leq Ch^{1/2-\alpha+1/2}\|\hat{v}\|_{L^2(\hat{T})}\leq Ch^{-\alpha}\|v\|_{L^2(T)}.
  \end{equation}

  When $h\leq 1$, we have
  \begin{equation*}
    \|v\|_{H^\alpha(T)}\leq Ch^{-\alpha}\|v\|_{L^2(T)}.
  \end{equation*}
\end{proof}
\end{appendix}




\end{document}